\documentclass[twoside,leqno,10pt, A4]{amsart}
\usepackage{amsfonts}
\usepackage{amsmath}
\usepackage{amscd}
\usepackage{amssymb}
\usepackage{amsthm}
\usepackage{amsrefs}
\usepackage{latexsym}
\usepackage{mathrsfs}
\usepackage{bbm}
\usepackage{amscd}
\usepackage{amssymb}
\usepackage{amsthm}
\usepackage{amsrefs}
\usepackage{latexsym}
\usepackage{mathrsfs}
\usepackage{bbm}
\usepackage{enumerate}
\usepackage{graphicx}
\allowdisplaybreaks
\usepackage{color}
\begin{document}

\newtheorem{theorem}[subsection]{Theorem}
\newtheorem{proposition}[subsection]{Proposition}
\newtheorem{lemma}[subsection]{Lemma}
\newtheorem{corollary}[subsection]{Corollary}
\newtheorem{conjecture}[subsection]{Conjecture}
\newtheorem{prop}[subsection]{Proposition}
\newtheorem{defin}[subsection]{Definition}

\numberwithin{equation}{section}
\newcommand{\mr}{\ensuremath{\mathbb R}}
\newcommand{\mc}{\ensuremath{\mathbb C}}
\newcommand{\dif}{\mathrm{d}}
\newcommand{\intz}{\mathbb{Z}}
\newcommand{\ratq}{\mathbb{Q}}
\newcommand{\natn}{\mathbb{N}}
\newcommand{\comc}{\mathbb{C}}
\newcommand{\rear}{\mathbb{R}}
\newcommand{\prip}{\mathbb{P}}
\newcommand{\uph}{\mathbb{H}}
\newcommand{\fief}{\mathbb{F}}
\newcommand{\majorarc}{\mathfrak{M}}
\newcommand{\minorarc}{\mathfrak{m}}
\newcommand{\sings}{\mathfrak{S}}
\newcommand{\fA}{\ensuremath{\mathfrak A}}
\newcommand{\mn}{\ensuremath{\mathbb N}}
\newcommand{\mq}{\ensuremath{\mathbb Q}}
\newcommand{\half}{\tfrac{1}{2}}
\newcommand{\f}{f\times \chi}
\newcommand{\summ}{\mathop{{\sum}^{\star}}}
\newcommand{\chiq}{\chi \bmod q}
\newcommand{\chidb}{\chi \bmod db}
\newcommand{\chid}{\chi \bmod d}
\newcommand{\sym}{\text{sym}^2}
\newcommand{\hhalf}{\tfrac{1}{2}}
\newcommand{\sumstar}{\sideset{}{^*}\sum}
\newcommand{\sumprime}{\sideset{}{'}\sum}
\newcommand{\sumprimeprime}{\sideset{}{''}\sum}
\newcommand{\sumflat}{\sideset{}{^\flat}\sum}
\newcommand{\shortmod}{\ensuremath{\negthickspace \negthickspace \negthickspace \pmod}}
\newcommand{\V}{V\left(\frac{nm}{q^2}\right)}
\newcommand{\sumi}{\mathop{{\sum}^{\dagger}}}
\newcommand{\mz}{\ensuremath{\mathbb Z}}
\newcommand{\leg}[2]{\left(\frac{#1}{#2}\right)}
\newcommand{\muK}{\mu_{\omega}}
\newcommand{\thalf}{\tfrac12}
\newcommand{\lp}{\left(}
\newcommand{\rp}{\right)}
\newcommand{\Lam}{\Lambda_{[i]}}
\newcommand{\lam}{\lambda}
\newcommand{\af}{\mathfrak{a}}
\newcommand{\sw}{S_{[i]}(X,Y;\Phi,\Psi)}
\newcommand{\lz}{\left(}
\newcommand{\pz}{\right)}
\newcommand{\bfrac}[2]{\lz\frac{#1}{#2}\pz}
\newcommand{\odd}{\mathrm{\ primary}}
\newcommand{\even}{\text{ even}}
\newcommand{\res}{\mathrm{Res}}
\newcommand{\sumn}{\sumstar_{(c,1+i)=1}  w\left( \frac {N(c)}X \right)}
\newcommand{\lab}{\left|}
\newcommand{\rab}{\right|}
\newcommand{\Go}{\Gamma_{o}}
\newcommand{\Ge}{\Gamma_{e}}
\newcommand{\M}{\widehat}
\def\su#1{\sum_{\substack{#1}}}

\theoremstyle{plain}
\newtheorem{conj}{Conjecture}
\newtheorem{remark}[subsection]{Remark}

\newcommand{\pfrac}[2]{\left(\frac{#1}{#2}\right)}
\newcommand{\pmfrac}[2]{\left(\mfrac{#1}{#2}\right)}
\newcommand{\ptfrac}[2]{\left(\tfrac{#1}{#2}\right)}
\newcommand{\pMatrix}[4]{\left(\begin{matrix}#1 & #2 \\ #3 & #4\end{matrix}\right)}
\newcommand{\ppMatrix}[4]{\left(\!\pMatrix{#1}{#2}{#3}{#4}\!\right)}
\renewcommand{\pmatrix}[4]{\left(\begin{smallmatrix}#1 & #2 \\ #3 & #4\end{smallmatrix}\right)}
\def\en{{\mathbf{\,e}}_n}

\newcommand{\ppmod}[1]{\hspace{-0.15cm}\pmod{#1}}
\newcommand{\ccom}[1]{{\color{red}{Chantal: #1}} }
\newcommand{\acom}[1]{{\color{blue}{Alia: #1}} }
\newcommand{\alexcom}[1]{{\color{green}{Alex: #1}} }
\newcommand{\hcom}[1]{{\color{brown}{Hua: #1}} }

\makeatletter
\def\widebreve{\mathpalette\wide@breve}
\def\wide@breve#1#2{\sbox\z@{$#1#2$}%
     \mathop{\vbox{\m@th\ialign{##\crcr
\kern0.08em\brevefill#1{0.8\wd\z@}\crcr\noalign{\nointerlineskip}%
                    $\hss#1#2\hss$\crcr}}}\limits}
\def\brevefill#1#2{$\m@th\sbox\tw@{$#1($}%
  \hss\resizebox{#2}{\wd\tw@}{\rotatebox[origin=c]{90}{\upshape(}}\hss$}
\makeatletter

\title[Lower Bounds on High Moments of Twisted Fourier coefficients of Modular Forms]{Lower Bounds on High Moments of Twisted Fourier coefficients of Modular Forms}

\author[P. Gao]{Peng Gao}
\address{School of Mathematical Sciences, Beihang University, Beijing 100191, China}
\email{penggao@buaa.edu.cn}

\author[L. Zhao]{Liangyi Zhao}
\address{School of Mathematics and Statistics, University of New South Wales, Sydney NSW 2052, Australia}
\email{l.zhao@unsw.edu.au}

\begin{abstract}
 For any large prime $q$, $x \geq 1$ with $x \leq q^{1/2}$ and any real $k\geq 2$, we prove a lower bound for the following $2k$-th moment
    \begin{equation*}
    \sum_{\substack{\chi \in X_q^*}} \Big| \sum_{n\leq x} \chi(n)\lambda(n)\Big|^{2k},
\end{equation*}
  where $X_q^*$ denotes the set of primitive Dirichlet characters modulo $q$ and $\lambda(n)$ the Fourier coefficients of a fixed holomorphic Hecke eigenform of weight $\kappa$ for the full modular group. The bound we obtain is sharp up to a constant factor when $q^{\varepsilon} \leq x \leq q^{1/2}$ for any fixed $\varepsilon>0$ under the generalized Riemann Hypothesis.
\end{abstract}

\maketitle



\section{Introduction}\label{sec1}

It goes without saying that character sums are extremely important in number theory and their utility cannot be over-stated.  In the breakthrough work \cite{harper2020moments}, A. J. Harper determined the order of magnitude of the high moments of Steinhaus or Rademacher random multiplicative functions.  The ideas used in \cite{harper2020moments}, together with those arising from \cite{Harper}, culminated in \cite{Harper23} in showing that the low moments of Dirichlet character sums have ``better than square root cancellation". More precisely, Harper \cite{Harper23} proved that if $q$ is a prime and $0\leq k\leq 1$, then for any $1 \leq x \leq q$,
\begin{equation*}
     \frac{1}{\varphi(q)} \sum_{\substack{\chi \in X^*_q}} \Big| \sum_{n\leq x} \chi(n) \Big|^{2k} \ll \bigg( \frac{x}{1+(1-k)\sqrt{\log\log \min(x, q/x)} }\bigg)^k, 
\end{equation*}
  where $X_q^*$ denotes the set of primitive Dirichlet characters modulo $q$ and $\varphi(q)$ is Euler's totient function. \newline

For higher moments,  B. Szab\'o \cite{Szab} applied his result on sharp upper bounds on shifted moments of Dirichlet $L$-functions on the critical line to show under the generalized Riemann hypothesis (GRH) that for a fixed real number $k>2$ and a large integer $q$, we have for $2 \leq Y \leq q$,
\begin{align*}
  \sum_{\chi\in X_q^*}\bigg|\sum_{n\leq Y} \chi(n)\bigg|^{2k} \ll_k \varphi(q) Y^k \left(\min \left(\log Y, \log \frac {2q}{Y} \right) \right)^{(k-1)^2}.
\end{align*}
  It was also shown in \cite[Theorem 1]{Szab24} that the above bounds are optimal for primes $q$. \newline

  Let $f$ be a fixed holomorphic Hecke eigenform of weight $\kappa$ for the full modular group $SL_2 (\mathbb{Z})$. We write the Fourier expansion of $f$ at infinity as
\[
f(z) = \sum_{n=1}^{\infty} \lambda (n) n^{(\kappa -1)/2} e(nz), \quad \mbox{where} \quad e(z) = \exp (2 \pi i z).
\]
  Motivated by the result of Szab\'o in \cite{Szab}, the authors studied upper bounds for high moments of sums involving with $\lambda (n)$ twisted by $\chi(n)$ to a fixed modulus.  More precisely, for positive real numbers $k$, $x$, we set
\begin{align*}
 S_{k}(q,x;f) := &\sum_{\chi \in X^*_q} \Big | \sum_{n \leq x}\chi(n)\lambda(n)\Big |^{2k}. 
\end{align*}
In \cite[Theorem 1.5]{G&Zhao24-12}, the authors show that assuming the truth of GRH, for large $q$, any $x \leq q$ and any real number $k> 2$,
\begin{align}
\label{Skupperbounds}
 S_{k}(q,x;f) \ll \varphi(q)x^k(\log q)^{(k-1)^2}.
\end{align}

  The aim of this paper is to obtain lower bounds for $S_{k}(q,x;f)$. Our result is as follows.
\begin{theorem}
\label{lowerboundsfixedmodmean}
With the notation as above, let $q$ be a large prime number. We have, for $1\leq x \leq q^{1/2}$ and any real number $k \geq 2$,
\begin{align*}
 S_{k}(q,x;f)  \gg \varphi(q)x^k(\log x)^{(k-1)^2}.
\end{align*}
  Here the implied constant depends on the Hecke eigenform.
\end{theorem}

Theorem \ref{lowerboundsfixedmodmean} holds unconditionally. Together with \eqref{Skupperbounds}, the following Corollary on the order of magnitude of $S_{k}(q,x;f)$ is immediate.
\begin{corollary}
\label{orderofmagfixedmodmean}
Assume GRH and use the same notation and conditions of Theorem~\ref{lowerboundsfixedmodmean}. Let $q^{\varepsilon} \leq x \leq q^{1/2}$ for any fixed $\varepsilon>0$. Then we have
\begin{align*}
 S_{k}(q,x;f) \asymp \varphi(q)x^k(\log x)^{(k-1)^2}.
\end{align*}
\end{corollary}

  Our proof of Theorem \ref{lowerboundsfixedmodmean} follows closely the treatments in \cite{Szab24}, which also used many techniques developed in the work of Harper \cites{Harper, Harper23, harper2020moments}.

\section{Preliminaries}
\label{sec 2}

In this section, we cite some results necessary in the proof of Theorem \ref{lowerboundsfixedmodmean}.

\subsection{Cusp form $L$-functions}
\label{sec:cusp form}

    We reserve the letter $p$ for a prime number throughout in this paper.  Recall that $f$ is a fixed holomorphic Hecke eigenform $f$ of weight $\kappa$ for the full modular group $SL_2 (\mathbb{Z})$. The associated modular $L$-function $L(s, f)$ for $\Re(s)>1$ is then defined as
\begin{align}
\label{Lphichi}
L(s, f ) &= \sum_{n=1}^{\infty} \frac{\lambda(n)}{n^s}
 = \prod_{p} \left(1 - \frac{\lambda (p)}{p^s}  + \frac{1}{p^{2s}}\right)^{-1}=\prod_{p} \left(1 - \frac{\alpha_p }{p^s} \right)^{-1}\left(1 - \frac{\beta_p }{p^s} \right)^{-1}.
\end{align}
 By Deligne's proof \cite{D} of the Weil conjecture, we know that
\begin{align}
\label{alpha}
|\alpha_{p}|=|\beta_{p}|=1, \quad \alpha_{p}\beta_{p}=1.
\end{align}
  It follows that $\lambda(n) \in \mr$ such that $\lambda (1) =1$ and
\begin{align}
\label{lambdabound}
\begin{split}
  |\lambda(n)| \leq d(n) \ll n^{\varepsilon}.
\end{split}
\end{align}
 where $d(n)$ is the number of positive divisors of $n$. \newline
 

  The symmetric square $L$-function $L(s, \operatorname{sym}^2 f)$ of $f$ is defined for $\Re(s)>1$ by
 (see \cite[p. 137]{iwakow} and \cite[(25.73)]{iwakow})
\begin{align}
\label{Lsymexp}
\begin{split}
 L(s, \operatorname{sym}^2 f)=& \prod_p(1-\alpha^2_pp^{-s})^{-1}(1-p^{-s})^{-1}(1-\beta^2_pp^{-s})^{-1} \\
 = & \zeta(2s) \sum_{n \geq 1}\frac {\lambda(n^2)}{n^s}=\prod_{p}\left( 1-\frac {\lambda(p^2)}{p^s}+\frac {\lambda(p^2)}{p^{2s}}-\frac {1}{p^{3s}} \right)^{-1}.
\end{split}
\end{align}
  It follows from a result of G. Shimura \cite{Shimura} that $L(s, \operatorname{sym}^2 f)$ has no pole at $s=1$. Moreover, the corresponding completed symmetric square $L$-function
\begin{align*}
 \Lambda(s, \operatorname{sym}^2 f)=& \pi^{-3s/2}\Gamma \Big( \frac {s+1}{2} \Big)\Gamma \Big(\frac {s+\kappa-1}{2}\Big) \Gamma \Big(\frac {s+\kappa}{2}\Big) L(s, \operatorname{sym}^2 f)
\end{align*}
  is entire and satisfies the functional equation $\Lambda(s, \operatorname{sym}^2 f)=\Lambda(1-s, \operatorname{sym}^2 f)$. \newline
  
   We derive from \eqref{Lphichi} and \eqref{Lsymexp} that
\begin{align}
\label{alphalambdarel}
\begin{split}
  \alpha_p+\beta_p= &\lambda(p), \\
  \alpha^2_p+\beta^2_p=& \lambda^2(p)-2=\lambda(p^2)-1.
\end{split}
\end{align}

Thus, it follows from the above that
\begin{align}
\label{sumlambdapsquare}
\begin{split}
  \lambda^2(p)=\lambda(p^2)+1.
\end{split}
\end{align}

\subsection{Sums over primes}
\label{sec2.1}

 We include in this section some asymptotic evaluations of various sums over primes.
\begin{lemma}
\label{RS}
 Let $x \geq 2$. We have, for some constant $b_1, b_2$,
\begin{align}
\label{merten}
\sum_{p\le x} \frac{1}{p} =& \log \log x + b_1+ O\Big(\frac{1}{\log x}\Big), \; \mbox{and} \\
\label{merten1}
\sum_{p\le x} \frac{\lambda^2(p)}{p} =& \log \log x + b_2+ O\Big(\frac{1}{\log x}\Big).
\end{align}
\end{lemma}
\begin{proof}
  The expressions in \eqref{merten} can be found in part (d) of \cite[Theorem 2.7]{MVa1} and the formula in \eqref{merten1} follows from \cite[Lemma 2.1]{GHH}.
\end{proof}

\begin{lemma}
\label{RS3}
  We have for $x \geq 2$ and $\alpha, \beta \in \mr$ with $0 \leq \beta \leq C/\log x$ for any positive constant $C$,
\begin{align}
\label{mertenstype}
  \sum_{p\leq x} \frac{\cos(\alpha \log p) }{p^{1+\beta}}=& \log |\zeta(1+1/\log x+\beta+i\alpha)| +O(1)
  \leq
\begin{cases}
\log\log x+O(1)            & \text{if} \;\;  0 \leq |\alpha| \leq 1/\log x  ,   \\
\log(1/|\alpha|)+O(1)        & \text{if} \;\;  1/\log x\leq |\alpha| \leq 10,   \\
\log\log|\alpha| + O(1) & \text{if} \;\;  10< |\alpha|.
\end{cases}
\end{align}
and
\begin{align} \label{mertenstypesympower}
  \sum_{p\leq x} \frac{\cos(\alpha \log p)\lambda(p^2) }{p^{1+\beta}}=& \log |L(1+1/\log x+\beta+i\alpha, \operatorname{sym}^2 f)| +O(1)
  \leq 3\log\log (|\alpha|+e^e)+O(1).
\end{align}
\end{lemma}
\begin{proof}
   The equality in  \eqref{mertenstype} is a special case given in \cite[Lemma 3.2]{Kou}, upon setting $f(n)=n^{-\beta}$ there. The estimations in
\eqref{mertenstype} follow from \cite[Lemma 2.9]{Munsch17}. \newline

   Similarly, the equality in  \eqref{mertenstypesympower} is a special case given in \cite[Lemma 3.2]{Kou}, upon setting $f(n)=\lambda(n^2)n^{-\beta}$ there. Note that although $f(p)=\lambda(p^2)p^{-\beta}$ may not be in the unit circle, it is bounded by \eqref{lambdabound} and the proof of \cite[Lemma 3.2]{Kou} carries over as well. In our case, the corresponding $L$-function becomes $L(1+1/\log x+\beta+i\alpha, \operatorname{sym}^2 f)\zeta^{-1}(2+2/\log x+2\beta+2i\alpha)$ by \eqref{Lsymexp} and $\log |\zeta(2+2/\log x+2\beta+2i\alpha)|=O(1)$. Note that the last estimation given in \eqref{mertenstypesympower} equals $O(1)$ when $|\alpha| \leq e^e$. In which case, the estimation follows by arguing similar to those given in the proof of \cite[Lemma 2.6]{G&Zhao24-12}. We may now assume $|\alpha|>e^e$ and follow the arguments in the proof of \cite[Theorem 6.7]{MVa1}.  Note first that by \eqref{alpha} and \eqref{Lsymexp}, we have, for $\sigma =\Re(s) \geq 1+1/\log (|t|+4)$ with $t=\Im(s)$,
\begin{align*}
  \Big |\frac{L'}{L}(s, \operatorname{sym}^2 f) \Big | \leq -3\frac {\zeta'(\sigma)}{\zeta(\sigma)} \ll \log (|t|+4),
\end{align*}
  where the last bound above follows from \cite[(6.9)]{MVa1}. Let $s_1=1+1/\log (|t|+4)+it$. The above gives that
\begin{align}
\label{logLderslargebound}
  \Big |\frac{L'}{L}(s_1, \operatorname{sym}^2 f) \Big | \ll \log (|t|+4).
\end{align}

     We deduce from this and \cite[(5.28)]{iwakow} that
\begin{align}
\label{sumoverzerosbound}
\begin{split}
   \sum_{\substack{\rho}} \Re\frac 1{s_1-\rho} \ll \log (|t|+4),
\end{split}
\end{align}
  where the sum is over those zeros $\rho$ of $L(s, \operatorname{sym}^2 f)$ with $|\rho-(3/2+it)| \leq 1$. Suppose
 that $1 \leq \sigma \leq 1+1/\log(|t|+4)$, then by \cite[(5.28)]{iwakow} again,
\begin{align}
\label{logLdersdiffbound}
\begin{split}
    \frac{L'}{L}(s, \operatorname{sym}^2 f)- \frac{L'}{L}(s_1, \operatorname{sym}^2 f)=\sum_{\rho}\Big (\frac 1{s-\rho}-\frac 1{s_1-\rho} \Big )+O(\log(|t|+4)).
\end{split}
\end{align}

From the known zero-free region of $L(s,  \operatorname{sym}^2 f)$ (see \cite[Theorem 5.10]{iwakow}), we get that $|s - \rho|\asymp |s_1 -\rho|$ for all zeros $\rho$ in the sum.  It follows that
\begin{align}
\label{sumoverzerosdiffbound}
\begin{split}
    \frac 1{s-\rho}-\frac 1{s_1-\rho} \ll \frac 1{|s_1-\rho|^2\log (|t|+4)} \ll \Re\frac 1{s_1-\rho}.
\end{split}
\end{align}

  We derive from \eqref{logLderslargebound}--\eqref{sumoverzerosdiffbound} that for $\Re(s) \geq 1$, we have
\begin{align}
\label{logLdersbound}
\begin{split}
    \frac{L'}{L}(s, \operatorname{sym}^2 f) \ll \log(|t|+4).
\end{split}
\end{align}

   Note further by \eqref{alpha} and \eqref{Lsymexp} that for $\sigma \geq 1+1/\log(|t|+4)$, we have
\begin{align}
\label{logLboundslarge}
 |\log L(s, \operatorname{sym}^2 f)| \leq 3|\log \zeta(\sigma)| \leq 3\log \log (|t|+4),
\end{align}
  where the last estimation above follows from \cite[Corollary 1.14]{MVa1}. \newline

   In particular, the above holds for $s=s_1$. From this and \eqref{logLdersbound}, for $1 \leq \Re(s) < 1+1/\log(|t|+4)$, we have
\begin{align}
\label{logLboundssmall}
  \log L(s, \operatorname{sym}^2 f)=\log L(s_1, \operatorname{sym}^2 f)+\int\limits^s_{s_1}\frac{L'}{L}(w, \operatorname{sym}^2 f)\dif w \leq  3\log \log (|t|+4)+O(1),
\end{align}
  where the path of the integration above is taken to be the line segment joining the endpoints.
  Now the second estimation given in \eqref{mertenstypesympower} follows readily from \eqref{logLboundslarge} and \eqref{logLboundssmall}.  This completes the proof of the lemma.
\end{proof}

\subsection{Mean Value Estimations} Let $(f(p))_{p \text{ prime}}$ be a sequence of independent random variables distributed uniformly on the unit circle in $\comc$.  A Steinhaus random multiplicative function $f$ is defined by setting $f(n):=\prod_{p^a \| n}f(p)^a$ for all natural numbers $n$.  Therefore, $f$ is a random function taking values in the complex unit circle and completely multiplicative. We denote the expectation by $\mathbb{E}$. \newline
 
  Our first result is taken from \cite[Lemma 4]{Szab24}.
\begin{lemma}
\label{evenmoment}
    Let $a_n$ and $c_n$ be two complex sequences. Let $\mathcal{P}$ be a finite set of primes and define $\Tilde{d}(n)=\sum_{d|n} \mathbf{1}(p|d \implies p\in \mathcal{P})$. For any integer $j\geq 0$ we have
    \begin{equation*}
        \mathbb{E}\Big| \sum_{n\leq x} c_nf(n)\Big|^2 \Big|\sum_{p\in \mathcal{P} }\frac{a_pf(p)}{p^{1/2}}+\frac{a_{p^2} f(p^2)}{p}\Big|^{2j}\ll \Big( \sum_{n\leq x}\Tilde{d}(n)|c_n|^2\Big) \cdot (j!)\cdot \Big(2\cdot \sum_{p\in \mathcal{P}}\frac{|a_p|^2}{p}+\frac{6|a_{p^2} |^2}{p^2} \Big)^j.
    \end{equation*}
\end{lemma}

The next result deals with the expectation of a certain random Euler product.
\begin{lemma}
\label{eulerproduct}
    Let $f(n)$ be a Steinhaus random multiplicative function, $\alpha,\beta,\sigma_1,\sigma_2\geq 0$ and $t_1,t_2 \in \rear$.  Suppose that $100(1+\max(\alpha^2,\beta^2))\leq z<y$. Then
    \begin{equation*}
    \begin{split}
        \mathbb{E}\prod_{z\leq p\leq y} & \Big|1-\frac{\alpha_pf(p)}{p^{1/2+\sigma_1+it_1}}\Big|^{-2\alpha}\Big|1-\frac{\beta_pf(p)}{p^{1/2+\sigma_1+it_1}}\Big|^{-2\alpha}
\Big|1-\frac{\alpha_pf(p)}{p^{1/2+\sigma_2+it_2}}\Big|^{-2\beta}\Big|1-\frac{\beta_pf(p)}{p^{1/2+\sigma_2+it_2}}\Big|^{-2\beta} \\
        =&\exp\bigg(\sum_{z \leq p\leq y}\frac{\alpha^2\lambda^2(p)}{p^{1+2\sigma_1} }+\frac{\beta^2\lambda^2(p)}{p^{1+2\sigma_2} }+\frac{2\alpha\beta \lambda^2(p)\cos((t_2-t_1)\log p)}{p^{1+\sigma_1+\sigma_2 } }+O\Big(\frac{\max(\alpha,\alpha^3,\beta,\beta^3)}{z^{1/2}}\Big)\bigg).
    \end{split}
    \end{equation*}
\end{lemma}
\begin{proof}
    This follows from Euler product result 1 in \cite{harper2020moments} and \cite[Lemma 1]{Szab24}, by noting that $(\alpha_p+\beta_p)^2=\lambda^2(p)$.
\end{proof}

The following lemma is taken from \cite[Theorem 5.4]{MVa1} and a version of Parseval’s identity for Dirichlet series.
\begin{lemma}
\label{parseval}
    Let $(a_n)_{n\geq 1}$ be a sequence of complex numbers and $F(s)=\sum_{n=1}^{\infty} a_nn^{-s}$ be the corresponding Dirichlet series. If $\sigma_c$ denotes its abscissa of convergence, then, for any $\sigma>\max(0,\sigma_c)$, we have
    \begin{equation*}
        \int\limits_{1}^{\infty} \frac{\big|\sum_{n\leq x}a_n\big|^2 }{x^{1+2\sigma }} \dif x=\frac{1}{2\pi}\int\limits_{-\infty}^{+\infty}\frac{|F(\sigma+it)|^2}{|\sigma+it|^2} \dif t.
    \end{equation*}
\end{lemma}

   Our next two results estimate the mean values of products involving $\lambda(n)$.
 \begin{lemma}
\label{lemlambdasquaresum1}
  For positive co-prime integers $c_1, c_2$, we have for $x \geq 2$ and some constant $C_1$, 
\begin{align}
\label{con1}
\begin{split}
     & \sum_{\substack{n \leq x }} |\lambda(c_1n)\lambda(c_2n)| \leq   C_1x(P_1(c_1c_2)+P_2(c_1c_2)),
\end{split}
\end{align}
  where for any positive integer $c$, 
\begin{align}
\label{P12}
\begin{split}
     P_1(c)=& \prod_{\substack{ p|c \\ p^{\nu_p} \| c}}\Big (|\lambda(p^{\nu_p})|+\frac {|\lambda(p^{\nu_p+1})|\cdot |\lambda(p)|}{p}+\frac {|\lambda(p^{\nu_p+2})|\cdot |\lambda(p^2)|}{p^{2}}+\cdots \Big ), \\
     P_2(c)=& \prod_{\substack{ p|c \\ p^{\nu_p} \| c}}\Big (|\lambda(p^{\nu_p})|+\frac {|\lambda(p^{\nu_p+1})|\cdot |\lambda(p)|}{p^{3/4}}+\frac {|\lambda(p^{\nu_p+2})|\cdot |\lambda(p^2)|}{p^{3/2}}+\cdots \Big ).
\end{split}
\end{align}
\end{lemma}
\begin{proof}
Dividing into dyadic blocks, to establish \eqref{con1}, it suffices to show that
\begin{align}
\label{con1dyadic}
\begin{split}
     & \sum_{\substack{x/2 < n \leq x }} |\lambda(c_1n)\lambda(c_2n)| \ll   x(P_1(c_1c_2)+P_2(c_1c_2)).
\end{split}
\end{align}

   Let $\Phi$ be a smooth, non-negative function compactly supported on $[1/4, 3/2]$ satisfying $\Phi(x) \leq 1$ for all $x$ and $\Phi(x) =1$
for $x\in [1/2,1]$, and recall that the Mellin transform ${\widehat \Phi}(s)$ of $\Phi$ is defined for any complex number $s$ by
\begin{equation*}
{\widehat \Phi}(s) = \int\limits_{0}^{\infty} \Phi(x)x^{s}\frac {\dif x}{x}.
\end{equation*}
    Note that integration by parts shows that $\widehat{\Phi}(s)$ is a function satisfying the bound
\begin{align}
\label{boundsforphi}
  \widehat{\Phi}(s) \ll \min (1, |s|^{-1}(1+|s|)^{-E}),
\end{align}
for any integers $E>0$ when $\Re(s) \in (\varepsilon, 10^{10})$. Here the implied constant depends on the endpoint of the
interval. \newline

In order to establish \eqref{con1dyadic}, we shall show 
\begin{align}
\label{con1dyadicsmooth}
\sum_{\substack{x/2 < n \leq x }} |\lambda(c_1n)\lambda(c_2n)|\Phi \Big(\frac {n}{x}\Big) \ll   x(P_1(c_1c_2)+P_2(c_1c_2)).
\end{align}  

Now the Mellin inversion leads to 
\begin{align}
\label{Perronprime1}
\begin{split}
  \sum_{\substack{x/2 < n \leq x }} |\lambda(c_1n)\lambda(c_2n)|\Phi \Big(\frac {n}{x}\Big) \leq  \sum_{n} |\lambda(c_1n)\lambda(c_2n)|\Phi \Big(\frac {n}{x} \Big) =&\frac 1{ 2\pi i}\int\limits_{(2)}F(s;c_1,c_2)\widehat{\Phi}(s) x^s \dif s. 
\end{split}
\end{align}
  where
\begin{align}
\label{F1}
\begin{split}
 F(s;c_1,c_2)=\sum_{\substack{n \geq 1}}\frac {|\lambda(c_1n)\lambda(c_2n)|}{n^{s}}.
\end{split}
\end{align}
  We write, for simplicity, $F(s)=F(s;1,1)$ and observe that $|\lambda(n)|^2=\lambda^2(n)$ as $\lambda(n)$ is real. We then note that 
\begin{align}
\label{Frel1}
\begin{split}
 F(s)\zeta(2s)=\zeta(s)L(s, \operatorname{sym}^2 f),
\end{split}
\end{align}
  with $L(s, \operatorname{sym}^2 f)$ defined in \eqref{Lsymexp}. \newline

  We thus deduce that for $\Re(s)>1$,
\begin{align*}
\begin{split}
    F(s;c_1,c_2) =& \prod_{\substack{ p|c_1c_2 \\ p^{\nu_p} \| c_1c_2}}\Big (|\lambda(p^{\nu_p})|+\frac {|\lambda(p^{\nu_p+1})|\cdot |\lambda(p)|}{p^s}+\frac {|\lambda(p^{\nu_p+2})|\cdot |\lambda(p^2)|}{p^{2s}}+\cdots \Big )\prod_{\substack{ p \nmid c_1c_2}}\Big (1+\frac {|\lambda(p)|^2}{p^s}+\frac {|\lambda(p^{2})|^2}{p^{2s}}+\cdots \Big ) \\
=& F(s)\prod_{\substack{ p|c_1c_2 \\ p^{\nu_p} \| c_1c_2}}\Big (|\lambda(p^{\nu_p})|+\frac {|\lambda(p^{\nu_p+1})|\cdot |\lambda(p)|}{p^s}+\frac {|\lambda(p^{\nu_p+2})|\cdot |\lambda(p^2)|}{p^{2s}}+\cdots \Big ) \frac{\zeta_{p}(2s)}{ \zeta_p(s)L_p(s, \operatorname{sym}^2 f)} \\
=& \frac{\zeta(s)L(s, \operatorname{sym}^2 f)}{\zeta(2s)}  \prod_{\substack{ p|c_1c_2 \\ p^{\nu_p} \| c_1c_2}}\Big (|\lambda(p^{\nu_p})|+\frac {|\lambda(p^{\nu_p+1})|\cdot |\lambda(p)|}{p^s}+\frac {|\lambda(p^{\nu_p+2})|\cdot |\lambda(p^2)|}{p^{2s}}+\cdots \Big ) \frac{\zeta_{p}(2s)}{ \zeta_p(s)L_p(s, \operatorname{sym}^2 f)} , 
\end{split}
\end{align*}
where $L_p$ denotes the local factor at the prime $p$ in the Euler product of $L$ for any $L$-function. Note that the above relation continues to hold for complex $s$ satisfying $\Re(s) \geq 1/2+\varepsilon$ for any $\varepsilon>0$ under GRH. \newline

  By \cite[(5.20)]{iwakow}, we have for $0 \leq \sigma \leq 1$,
\begin{align}
\label{zetaest1}
\begin{split}
    \frac {s-1}{s+1}\zeta(s) \ll (|s|+1)^{(1-\sigma)/2+\varepsilon}, \quad L(s, \operatorname{sym}^2 f)\ll (|s|+1)^{3(1-\sigma)/2+\varepsilon}.
\end{split}
\end{align}

  Moreover, for $\Re(s) \geq 1/2+\varepsilon$, from \cite[Theorem 6.7]{MVa1}, we get
\begin{align}
\label{zeta2sbound1}
\begin{split}
   \frac {1}{\zeta(2s)} \ll &
\begin{cases}
   |2s-1|, \quad |\Im(2s)| < 7/8, \\
   \log (|\Im(2s)|+4), \quad |\Im(2s)| \geq 7/8.
\end{cases}
\end{split}
\end{align}

Hence from \eqref{Frel1}, \eqref{zetaest1} and \eqref{zeta2sbound1}, for $\sigma=\Re(s) \geq 1/2+\varepsilon$,
\begin{align}
\label{Fbound1}
\begin{split}
   \frac {s-1}{s+1}F(s) \ll &
\begin{cases}
   |2s-1|(|s|+1)^{2(1-\sigma)+\varepsilon}, \quad |\Im(2s)| < 7/8, \\
   \log (|\Im(2s)|+4)(|s|+1)^{2(1-\sigma)+\varepsilon}, \quad |\Im(2s)| \geq 7/8.
\end{cases}
\end{split}
\end{align}

  We now shift the contour of the integral in \eqref{Perronprime1} to $\Re(s)=1-\varepsilon$ to encounter a simple pole at $s=1$ of $\zeta(s)$. By \eqref{Frel1}, the residue equals
\begin{align}
\label{ResPerron1}
\begin{split}
    x\prod_{\substack{ p|c_1c_2 \\ p^{\nu_p} \| c_1c_2}} & \Big (|\lambda(p^{\nu_p})|+\frac {|\lambda(p^{\nu_p+1})|\cdot |\lambda(p)|}{p}+\frac {|\lambda(p^{\nu_p+2})|\cdot |\lambda(p^2)|}{p^{2}}+\cdots \Big ) \frac{\zeta_{p}(2)}{ \zeta_p(1)L_p(1, \operatorname{sym}^2 f)}  \frac {L(1, \operatorname{sym}^2 f)}{\zeta(2)}{\widehat \Phi}(1) \\
\ll &  x\prod_{\substack{ p|c_1c_2 \\ p^{\nu_p} \| c_1c_2}}\Big (|\lambda(p^{\nu_p})|+\frac {|\lambda(p^{\nu_p+1})|\cdot |\lambda(p)|}{p}+\frac {|\lambda(p^{\nu_p+2})|\cdot |\lambda(p^2)|}{p^{2}}+\cdots \Big ) =  xP_1(c_1c_2),
\end{split}
\end{align}
  where we note that by \eqref{alpha} and \eqref{Lsymexp}, we have for any prime $p$ and $\sigma \geq 1/2+\varepsilon$, 
\begin{align}
\label{zetaprodbound}
\begin{split}
   \frac{\zeta_{p}(2s)}{ \zeta_p(s)L_p(s, \operatorname{sym}^2 f)}  =(1+\frac 1{p^s})^{-1}(1-\alpha^2_pp^{-s})(1-p^{-s})(1-\beta^2_pp^{-s}) \ll 1.
\end{split}
\end{align} 

    We further apply \eqref{boundsforphi} and \eqref{Fbound1} to see that the integral on the new line is 
\begin{align}
\label{Perronprimeint2}
\begin{split}
   \ll & \int\limits_{(1-\varepsilon)}F(s;c_1,c_2)\widehat{\Phi}(s) x^s \dif s \\
   \ll &  x\prod_{\substack{ p|c_1c_2 \\ p^{\nu_p} \| c_1c_2}}\Big (|\lambda(p^{\nu_p})|+\frac {|\lambda(p^{\nu_p+1})|\cdot |\lambda(p)|}{p^{1-\varepsilon}}+\frac {|\lambda(p^{\nu_p+2})|\cdot |\lambda(p^2)|}{p^{2-2\varepsilon}}+\cdots \Big ) \int\limits_{(1-\varepsilon)}\Big|F(s)\Big | \Big | \widehat{\Phi}(s)\Big | \Big | \dif s\Big |  \\
   \ll & xP_2(c_1c_2).
\end{split}
\end{align}

Hence, we deduce from \eqref{Perronprime1}, \eqref{ResPerron1} and \eqref{Perronprimeint2} with $\varepsilon = 1/4$ that for some constant $C_1$, \eqref{con1dyadicsmooth} holds. This completes the proof of the lemma.
\end{proof}

\begin{lemma}
\label{lemlambdasquaresum}
  We have, for $x \geq 2$,
\begin{align}
\label{con0}
    \sum_{\substack{n \leq x \\ c|n}} |\lambda(n)|^2=  \frac {x}{c} \frac{ L(1, \operatorname{sym}^2 f)}{\zeta(2)}\prod_{\substack{ p|c \\ p^{\nu_p} \| c}}\Big ( \sum_{j=0}^{\infty} \frac{|\lambda(p^{\nu_p+j})|^2}{p^j} \Big ) \frac{\zeta_{p}(2) }{ \zeta_p(1)L_p(1, \operatorname{sym}^2 f) } +O(x^{3/4+\varepsilon}),
\end{align}
where the implied constant in the $O$-term above depends on $f$ and $\varepsilon$, but not on $c$.
\end{lemma}
\begin{proof}
  We apply Perron's formula as given in \cite[Corollary 5.3]{MVa1} to see that
\begin{align}
\label{Perronprime}
\begin{split}
    &\sum_{\substack{n \leq x \\ c|n}} |\lambda(n)|^2=  \frac 1{ 2\pi i}\int\limits_{1+1/\log x-iT}^{1+1/\log x+iT}\Big(\sum_{\substack{n \geq 1 \\ c|n}}\frac {|\lambda(n)|^2}{n^s}\Big) \frac{x^s}{s} \dif s +R,
\end{split}
\end{align}
  where
\begin{align*}
\begin{split}
  R \ll &  \sum_{\substack{x/2< n <2x  \\  n \neq x \\ c|n  }}|\lambda(n)|^2\min \left( 1, \frac {x}{T|n-x|} \right) +\frac {4^{1+1/\log x}+x^{1+1/\log x}}{T}\sum_{\substack{n \geq 1 \\ c|n}}\frac {|\lambda(n)|^2}{n^{1+1/\log x}}.
\end{split}
\end{align*}

  We now apply the estimation given in \eqref{lambdabound} to see that
\begin{align}
\label{R1}
\begin{split}
   R \ll & x^{\varepsilon} \sum_{\substack{x/2< n <2x  \\  n \neq x   }}\min \left( 1, \frac {x}{T|n-x|} \right)+\frac {x^{1+\varepsilon}}{T}\sum_{\substack{n \geq 1 }}\frac {1}{n^{1+1/\log x}} \ll  x^{\varepsilon}\big( 1+\frac {x}{T} \big),
\end{split}
\end{align}
  where the last estimation above follows from the bound given for $R_1$ on \cite[p. 401]{MVa1}. \newline

  We deduce from \eqref{Frel1} that for $\Re(s)>1$,
\begin{align}
\label{Lambdansquare}
\begin{split}
 \sum_{\substack{n \geq 1 \\ c|n}}\frac {|\lambda(n)|^2}{n^{s}}=& \frac 1{c^s}\sum_{\substack{n \geq 1 }}\frac {|\lambda(cn)|^2}{n^{s}} \\
=& \frac 1{c^s}\prod_{\substack{ p|c \\ p^{\nu_p} \| c}}\Big (|\lambda(p^{\nu_p})|^2+\frac {|\lambda(p^{\nu_p+1})|^2}{p^s}+\frac {|\lambda(p^{\nu_p+2})|^2}{p^{2s}}+\cdots \Big )\prod_{\substack{ p \nmid c}}\Big (1+\frac {|\lambda(p)|^2}{p^s}+\frac {|\lambda(p^{2})|^2}{p^{2s}}+\cdots \Big ) \\
=& \frac {F(s)}{c^s}\prod_{\substack{ p|c \\ p^{\nu_p} \| c}}\Big ( \sum_{j=0}^{\infty} \frac{|\lambda(p^{\nu_p+j})|^2}{p^{js}} \Big ) \frac{\zeta_p(2s)}{\zeta_p(s)L_p(s, \operatorname{sym}^2 f)}.
\end{split}
\end{align}
Note that the above relation continues to hold for all complex $s$. \newline

Shifting the contour of the integral in \eqref{Perronprime} to $\Re(s)=1/2+\varepsilon$ to pick up a simple pole at $s=1$ of $\zeta(s)$. By \eqref{Frel1}, we see that the corresponding residue equals
\begin{align*}
\begin{split}
    \frac x{c} \frac{ L(1, \operatorname{sym}^2 f)}{\zeta(2)} \prod_{\substack{ p|c \\ p^{\nu_p} \| c}}\Big ( \sum_{j=0}^{\infty} \frac{|\lambda(p^{\nu_p+j})|^2}{p^{j}} \Big ) \frac {\zeta_p(2) }{\zeta_p(1)L_p(1, \operatorname{sym}^2 f)}.
\end{split}
\end{align*}

Therefore,
\begin{align}
\label{Perronprimeint1}
\begin{split}
 \frac 1{ 2\pi i}  \int\limits_{1+1/\log x-iT}^{1+1/\log x+iT} & \Big(\sum_{\substack{n \geq 1 \\ c|n}}\frac {|\lambda(n)|^2}{n^s}\Big) \frac{x^s}{s} \dif s \\
& = \frac x{c} \frac{ L(1, \operatorname{sym}^2 f)}{\zeta(2)} \prod_{\substack{ p|c \\ p^{\nu_p} \| c}} \Big ( \sum_{j=0}^{\infty} \frac{|\lambda(p^{\nu_p+j})|^2}{p^{j}} \Big ) \frac {\zeta_p(2) }{\zeta_p(1)L_p(1, \operatorname{sym}^2 f)} \\
& \hspace*{2cm} +O\Bigg(\int\limits_{1/2+\varepsilon-iT}^{1/2+\varepsilon+iT}\Big(\sum_{\substack{n \geq 1 \\ c|n}}\frac {|\lambda(n)|^2}{n^s}\Big) \frac{x^s}{s} \dif s+\int\limits_{1/2+\varepsilon \pm iT}^{1+1/\log x\pm iT}\Big(\sum_{\substack{n \geq 1 \\ c|n}}\frac {|\lambda(n)|^2}{n^s}\Big) \frac{x^s}{s} \dif s \Bigg).
\end{split}
\end{align}

 We estimate the $O$-term in \eqref{Perronprimeint1} by applying \eqref{lambdabound} and \eqref{zetaprodbound} to see that for $\sigma \geq 1/2+\varepsilon$,
\begin{align*}
\begin{split}
  & \frac {1}{c^s}\prod_{\substack{ p|c \\ p^{\nu_p} \| c}} \Big ( \sum_{j=0}^{\infty} \frac{|\lambda(p^{\nu_p+j})|^2}{p^{js}} \Big ) \frac{\zeta_p(2s)}{\zeta_p(s)L_p(s, \operatorname{sym}^2 f)} \ll  \frac {1}{c^{1/2}}\prod_{\substack{ p|c \\ p^{\nu_p} \| c}} \Big ( \sum_{j=0}^{\infty} \frac{|\lambda(p^{\nu_p+j})|^2}{p^{j/2}} \Big ) \ll \frac {1}{c^{1/2}}\prod_{\substack{ p|c \\ p^{\nu_p} \| c}} \Big ( \sum_{j=0}^{\infty} \frac{(\nu_p+j+1)^2}{p^{j/2}} \Big ).
\end{split}
\end{align*}

  Using $(\nu_p+j)/(\nu_p+1) \leq j$ for all $j \geq 1$, we deduce from the above that for $\sigma \geq 1/2+\varepsilon$,
\begin{align}
\label{dfactorboundprod}
\begin{split}
 & \frac {1}{c^s}\prod_{\substack{ p|c \\ p^{\nu_p} \| c}} \Big ( \sum_{j=0}^{\infty} \frac{|\lambda(p^{\nu_p+j})|^2}{p^{js}} \Big ) \frac{\zeta_p(2s)}{\zeta_p(s)L_p(s, \operatorname{sym}^2 f)}  \ll  \frac {1}{c^{1/2}} \prod_{\substack{ p|c}}(\nu_p+1)^2 \cdot \prod_{\substack{ p|c}}\Big (1+\frac {2^2}{p^{1/2}}+\frac {3^2}{p}+\cdots \Big )  
\ll  \frac {\omega(c)^2D^{\omega(c)}}{c^{1/2}}, 
\end{split}
\end{align}
  where $D$ is a constant and $\omega(c)$ denotes the number of primes dividing $c$. \newline
  
  Note that we have from \cite[Theorem 2.10]{MVa1} that for $n \geq 3$,
\begin{align*}
\omega(n) \ll \frac {\log n}{\log \log n}.
\end{align*}

  It follows from the above and \eqref{dfactorboundprod} that
\begin{align}
\label{dfactorboundfurthersimplified}
\begin{split}
 & \frac {1}{c^s}\prod_{\substack{ p|c \\ p^{\nu_p} \| c}} \Big ( \sum_{j=0}^{\infty} \frac{|\lambda(p^{\nu_p+j})|^2}{p^{js}} \Big ) \frac{\zeta_p(2s)}{\zeta_p(s)L_p(s, \operatorname{sym}^2 f)}  \ll  1. 
\end{split}
\end{align}  

 We deduce from \eqref{Fbound1}, \eqref{Lambdansquare} and \eqref{dfactorboundfurthersimplified} that
\begin{align}
\label{Intnewlineest}
\begin{split}
     \int\limits_{1/2+\varepsilon-iT}^{1/2+\varepsilon+iT}\Big(\sum_{\substack{n \geq 1 \\ c|n}}\frac {|\lambda(n)|^2}{n^s}\Big) \frac{x^s}{s} \dif s \ll x^{1/2+\varepsilon}T^{1+\varepsilon}.
\end{split}
\end{align}

   Similarly,
\begin{align}
\label{Intverlineest}
\begin{split}
    \int\limits_{1/2+\varepsilon \pm iT}^{1+1/\log x\pm iT}\Big(\sum_{\substack{n \geq 1 \\ c|n}}\frac {|\lambda(n)|^2}{n^s}\Big) \frac{x^s}{s} \dif s \ll \int\limits^{1+1/\log x}_{1/2+\varepsilon}\frac {x^{\sigma}T^{2(1-\sigma)+\varepsilon}}{T} \dif \sigma \ll x^{1/2}T^{\varepsilon}+\frac{xT^{\varepsilon}}{
T}.
\end{split}
\end{align}

   We conclude from \eqref{Perronprime}, \eqref{R1}, \eqref{Perronprimeint1}, \eqref{Intnewlineest} and \eqref{Intverlineest} that
\begin{align*}
   \sum_{\substack{n \leq x \\ c|n}} |\lambda(n)|^2= \frac x{c} \frac{ L(1, \operatorname{sym}^2 f)}{\zeta(2)} \prod_{\substack{ p|c \\ p^{\nu_p} \| c}} \Big ( \sum_{j=0}^{\infty} \frac{|\lambda(p^{\nu_p+j})|^2}{p^{j}} \Big ) \frac {\zeta_p(2) }{\zeta_p(1)L_p(1, \operatorname{sym}^2 f)} +O\Big(x^{\varepsilon}\big( 1+\frac {x}{T} \big)+ x^{1/2+\varepsilon}T^{1+\varepsilon}+\frac{x^{1+\varepsilon} T^{\varepsilon}}{T} \Big).
\end{align*}

Setting $T=x^{1/4}$ leads to \eqref{con0}, completing the proof of the lemma.
\end{proof}

\section{Outline of the proof of Theorem \ref{lowerboundsfixedmodmean}} \label{outline}

 As mentioned in the Introduction, the proof of Theorem \ref{lowerboundsfixedmodmean} follows closely the treatments in \cite{Szab24}.  Indeed, the motivations of the definitions of various quantities below can be found in \cite{Szab24}.  We write $y=x^{1/C_0}$, where $C_0$ is a large absolute constant whose value depends on $k$ only, to be specified later. We then define a subdivision of the interval $[1,y]$ with $1=y_0<y_1<\ldots <y_M=y$  recursively by setting $y_M=y$ and $y_{m-1} =y_m^{1/20}$ for any $2\leq m\leq M$. We choose $M$ such that $y_1$ lies in $\big[y^{1/(20(\log \log y)^2)}, y^{1/(\log \log y)^2}\big]$. \newline

 We further define parameters $J_m$ for $1\leq m\leq M$ such that $J_1=(\log \log y)^{3/2}$, $J_M=C_0/(10^5k)$, and that $J_m=J_M+M-m$ for $2\leq m\leq M-1$. We take $C_0$ large enough to ensure that $J_M$ is an integer such that $J_M\geq \exp(10^4 k^2)$.  Using these notations, we have
\begin{equation}
\label{shortpolynomial}
    \prod_{m=1}^M y_m^{10^4kJ_m}<x.
\end{equation}
 
  For $1\leq m\leq M$ and any integers $|l|\leq (\log y)/2$, we define for any $\chi$ modulo $q$,   
\begin{align*}
    D_{m,l}(\chi)=& \sum_{y_{m-1}< p\leq y_m } \frac{(\alpha_p+\beta_p)\chi(p)}{p^{1/2+il/\log y}} +\frac{(\alpha^2_p+\beta^2_p)\chi(p)^2}{2p^{1+2il/\log y}}, \quad    R_{m,l}(\chi) = \bigg(\sum_{0\leq j \leq J_m} \frac{(k-1)^j}{j!}\big(\Re D_{m,l}(\chi)\big)^j\bigg)^2 \quad \mbox{and} \\
    R(\chi)= & \sum_{|l|\leq (\log y)/2} \prod_{m=1}^M  R_{m,l}(\chi) \\
    & = \sum_{|l|\leq (\log y)/2} \prod_{m=1}^M  \bigg(\sum_{0 \leq j \leq J_m} \frac{(k-1)^j}{j!}\bigg(\Re \sum_{y_{m-1}< p\leq y_m } \frac{(\alpha_p+\beta_p)\chi(p)}{p^{1/2+il/\log y}} +\frac{(\alpha^2_p+\beta^2_p)\chi(p)^2}{2p^{1+2il/\log y}} \bigg)^j\bigg)^2.
\end{align*}
 
  We also define the corresponding quantities for the random multiplicative function $f(n)$, 
\begin{align*}
    D_{m,l}(f)= &\sum_{y_{m-1}< p\leq y_m } \frac{(\alpha_p+\beta_p)f(p)}{p^{1/2+il/\log y}} +\frac{(\alpha^2_p+\beta^2_p)f(p)^2}{2p^{1+2il/\log y}},
\end{align*}
  and $R_{m,l}(f)$ and $R(f)$ in the same way based on $R_{m,l}(\chi)$ and $R(\chi)$, respectively. \newline

  We note that the quantity $R(\chi)$ is introduced to approximate $|\sum_{n\leq x} \chi(n)\lambda(n)|^{2(k-1)}$.  Hölder's inequality with exponents $k$ and $k/(k-1)$ reveals that
\begin{equation*}
    \bigg(\frac{1}{\varphi(q)}\sum_{\substack{\chi \text{ mod }q\\ \chi\neq \chi_0}}\bigg|\sum_{n\leq x} \chi(n)\lambda(n)\bigg|^{2k}\bigg)^{1/k} \bigg( \frac{1}{\varphi(q)}\sum_{\substack{\chi \text{ mod }q\\ \chi\neq \chi_0}} R(\chi)^{k/(k-1)}\bigg)^{(k-1)/k}\geq \frac{1}{\varphi(q)} \sum_{\substack{\chi \text{ mod }q\\ \chi\neq \chi_0}} \bigg|\sum_{n\leq x}\chi(n)\lambda(n)\bigg|^2 R(\chi).
\end{equation*}
  
Consequently, to prove Theorem \ref{lowerboundsfixedmodmean}, it suffices to establish the following two propositions. 
\begin{proposition}
\label{p1}
 With the notation as above, for $C_0$ large enough, we have
\begin{equation}
\label{lowerbound}
    \frac{1}{\varphi(q)} \sum_{\substack{\chi \text{ mod }q\\ \chi\neq \chi_0}} \bigg|\sum_{n\leq x}\chi(n)\lambda(n)\bigg|^2 R(\chi)\gg_k x(\log y)^{k^2-1}.
\end{equation}
\end{proposition}

\begin{proposition}
\label{p2}
   With the notation as above, for $C_0$ large enough, we have
\begin{equation*}
        \frac{1}{\varphi(q)}\sum_{\substack{\chi \text{ mod }q\\ \chi\neq \chi_0}} R(\chi)^{k/(k-1)}\ll_k (\log y)^{k^2+1}.
\end{equation*}
\end{proposition}

The remainder of the paper is devoted to the proofs of these propositions.

\section{Proof of Proposition \ref{p1}}

 We first show that adding $\big|\sum_{n\leq x}\chi_0(n) \lambda(n) \big|^2R(\chi_0)/\varphi(q)$ to the left-hand side of \eqref{lowerbound} does not affect the desired bound. Applying \eqref{alpha} and summing trivially, we see that
\begin{equation*}
    |D_{m,l}(\chi_0)|= \Big| \sum_{y_{m-1}< p\leq y_m } \frac{(\alpha_p+\beta_p)\chi_0(p)}{p^{1/2+il/\log y}} +\frac{(\alpha^2_p+\beta^2_p)\chi_0(p)^2}{2p^{1+2il/\log y}} \Big| \leq y_m.
\end{equation*}
   It follows that
\begin{equation*}
    R_{m,l}(\chi_0)=\bigg(\sum_{j=0 }^{J_m} \frac{(k-1)^j}{j!}\big(\Re D_{m,l}(\chi_0)\big)^j\bigg)^2\leq (ky_m)^{2J_m}.
\end{equation*}
   We then deduce from \cite[Theorem 2.3]{MVa1} and \eqref{shortpolynomial} that, as $x\leq q^{1/2}$,
\begin{equation*}
    \Big|\sum_{n\leq x}\chi_0(n)\lambda(n)\Big|^2R(\chi_0)\ll \Big|\sum_{n\leq x}d(n)\Big|^2R(\chi_0)\ll x^2(\log x)^2(\log y)\prod_{m=1}^M (ky_m)^{2J_m} <x^{2+1/10} \ll \varphi(q)x^{1/5}.
\end{equation*}
Thus the above shows that removing the restriction of $\chi \neq \chi_0$ on the left-hand side of \eqref{lowerbound} introduces a negligible error and it suffices to show that
\begin{equation*}
   \frac{1}{\varphi(q)} \sum_{\chi \text{ mod }q} \bigg|\sum_{n\leq x}\chi(n)\lambda(n) \bigg|^2 R(\chi)\gg_k x(\log y)^{k^2-1}.
\end{equation*}

  By expanding the brackets in the definition of $R(\chi)$ for any $\chi$ modulo $q$, we may write for some suitable complex coefficients $a_{n_1,n_2}$,
\begin{equation}
\label{expand}
    \bigg|\sum_{n\leq x}\chi(n)\lambda(n)\bigg|^2 R(\chi)=\sum_{|l|\leq (\log y)/2}\sum_{\substack{n_1,n_2\leq N \\ m_1,m_2\leq x}} a_{n_1,n_2}\frac{\chi(m_1n_1)\bar{\chi}(m_2n_2)\lambda(m_1) \lambda(m_2)}{n_1^{1/2+il/\log y}n_2^{1/2-il/\log y} }.
\end{equation}

By \eqref{shortpolynomial},
\begin{equation*}
    xN\leq x \prod_{m=1}^M y_m^{4J_m}\leq x^{3/2}<q.
\end{equation*}
Hence when averaging \eqref{expand} over $\chi$ modulo $q$, the orthogonality relation implies that the only non-zero contribution comes from the diagonal terms $m_1n_1=m_2n_2$. Thus,
\begin{equation*}
    \frac{1}{\varphi(q)}\sum_{\chi\, \text{mod}\, q}\bigg|\sum_{n\leq x}\chi(n)\lambda(n)\bigg|^2 R(\chi)=\mathbb{E}\sum_{|l|\leq \log y/2}\sum_{\substack{n_1,n_2\leq N \\ m_1,m_2\leq x}} a_{n_1,n_2}\frac{f(m_1n_1)\bar{f}(m_2n_2)\lambda(m_1) \lambda(m_2)}{n_1^{1/2+il/\log y}n_2^{1/2-il/\log y} }=\mathbb{E}\bigg|\sum_{n\leq x} f(n)\lambda(n) \bigg|^2 R(f).
\end{equation*}

Therefore it suffices to show that
\begin{equation*}
    \mathbb{E}\bigg|\sum_{n\leq x} f(n)\lambda(n)\bigg|^2 R(f)\gg_k x(\log y)^{k^2-1}.
\end{equation*}

  We deduce similar to the arguments used in the proof of \cite[Proposition 3.1]{Szab24} that in order to prove the above bound, it suffices to establish the following two propositions.
\begin{proposition}
\label{p3}
  With the notation as above, we have
\begin{equation*}
  \sum_{|l|\leq \log y/2} \mathbb{E}  \bigg|\sum_{n\leq x}f(n)\lambda(n)\bigg|^2\exp\bigg(2(k-1)\Re\sum_{p\leq y}\frac{(\alpha_p+\beta_p)f(p)}{p^{1/2+il/\log y}} +\frac{(\alpha^2_p+\beta^2_p)f(p)^2}{2p^{1+2il/\log y}} \bigg)\geq e^{O(k^2\log \log k)} x(\log y)^{k^2-1}. 
\end{equation*}
\end{proposition}
\begin{proposition}
\label{p4}
 With the notation as above, we set
\begin{equation*}
    \text{Err}_{m,l}(f) =\exp\Big(2(k-1)\Re D_{m,l}(f)\Big)-R_{m,l}(f)=\sum_{\substack{j_1,j_2\geq 0\\ \max(j_1,j_2)>J_m }} \frac{(k-1)^{j_1+j_2}}{j_1!j_2!}(\Re D_{m,l}(f))^{j_1+j_2}.
\end{equation*}
  We have for any $1\leq m\leq M$, 
\begin{equation*}
    \sum_{|l|\leq \log y/2} \mathbb{E}  \bigg|\sum_{n\leq x}f(n)\lambda(n)\bigg|^2\exp\Big(2(k-1)\sum_{\substack{m'=1\\ m'\neq m} }^M\Re D_{m',l}(f)\Big) |\text{Err}_{m,l}(f)|\leq e^{O(k^4)}e^{-J_m} \frac{\log x}{\log y}x(\log y)^{k^2-1}.
\end{equation*}
\end{proposition}

   In the remainder of this section, we prove these two propositions. 

\subsection{Proof of Proposition \ref{p3}}
Note that by \eqref{alpha},
\begin{equation*}
    \frac{(\alpha_p+\beta_p)f(p)}{p^{1/2+il/\log y}} +\frac{(\alpha^2_p+\beta^2_p)f(p)^2}{2p^{1+2il/\log y}}=-\log\Big(1-\frac{\alpha_pf(p)}{p^{1/2+il/\log y}}\Big)-\log\Big(1-\frac{\beta_pf(p)}{p^{1/2+il/\log y}}\Big)+O\big(p^{-3/2}\big).
\end{equation*}

  As $\sum_{p}p^{-3/2}\ll 1$,
\begin{align}
\label{appending}
\begin{split}
    \sum_{|l|\leq \log y/2} & \mathbb{E} \bigg|\sum_{n\leq x}f(n)\lambda(n)\bigg|^2 \exp\bigg(2(k-1) \Re \sum_{p\leq y } \frac{(\alpha_p+\beta_p)f(p)}{p^{1/2+il/\log y}} +\frac{(\alpha^2_p+\beta^2_p)f(p)^2}{2p^{1+2il/\log y}}\bigg) \\
     =& e^{O(k)}  \sum_{|l|\leq \log y/2}  \mathbb{E} \bigg|\sum_{n\leq x}f(n)\lambda(n)\bigg|^2 \prod_{p\leq y}\bigg| 1-\frac{\alpha_pf(p)}{p^{1/2+il/\log y}}\bigg|^{-2(k-1)}\bigg| 1-\frac{\beta_pf(p)}{p^{1/2+il/\log y}}\bigg|^{-2(k-1)}, \\
=& e^{O(k)}  \sum_{|l|\leq \log y/2}  \mathbb{E}\mathbb{E}^{(y)} \bigg|\sum_{n\leq x}f(n)\lambda(n)\bigg|^2|F_y(1/2+il/\log y)|^{2(k-1)},
\end{split}
\end{align}
 where $\mathbb{E}^{(y)}$ denotes the conditional expectation with respect to $(f(p))_{p\leq y}$ and
\begin{equation*}
    F_y(s)=\prod_{p\leq y}\Big|1-\frac{\alpha_pf(p)}{p^s}\Big|^{-1}\Big|1-\frac{\beta_pf(p)}{p^s}\Big|^{-1}.
\end{equation*}

   The next lemma allows us to estimate $\mathbb{E}^{(y)} \bigg|\sum_{n\leq x}f(n)\lambda(n)\bigg|^2$ from below. 
\begin{lemma}
\label{condition_y}
    With the notation as above, assume that $y<x^{1/10}$ is large. Then, for any $\beta>0$, 
     \begin{equation*}
        \mathbb{E}^{(y)} \bigg|\sum_{n\leq x}f(n)\lambda(n)\bigg|^2\gg \frac{x}{\log y} \bigg[  \int\limits_{-1/2}^{1/2} |F_{y}(\tfrac12+\beta+it)|^2 \dif t-x^{-\beta/1000} \int\limits_{-\infty}^{\infty} |F_{y}(\tfrac12+\tfrac{\beta}{2}+it)|^2 \frac{\dif t}{1/4+t^2}\bigg],
    \end{equation*}
where the implied constant is absolute.
\end{lemma}
\begin{proof}
  Denote $P^{-}(n)$ and $P^{+}(n)$ be the smallest and largest prime factor of a positive integer $n$. We then have
\begin{equation*}
    \mathbb{E}^{(y)} \bigg|\sum_{n\leq x}f(n)\lambda(n)\bigg|^2=\mathbb{E}^{(y)} \bigg|\sum_{\substack{n\leq x\\ P^{-}(n)>y } }f(n) \lambda(n)\sum_{\substack{m\leq x/n \\ P^{+}(m)\leq y }}f(m) \lambda(m) \bigg|^2.
\end{equation*}
Set
\[ c_n= \sum_{\substack{m\leq x/n \\ P^+(m)\leq y } }f(m)\lambda(m)\]
so that $c_n$ is determined by $(f(p))_{p\leq y}$. Therefore it may be regarded as a fixed quantity when taking $\mathbb{E}^{(y)}$. Furthermore, for $n_1, n_2$ with $P^{-}(n_1n_2)>y$, the orthogonality relation gives that $\mathbb{E}^{(y)} f(n_1)\overline{f(n_2)}=\mathbf{1}(n_1=n_2)$. It follows that
\begin{equation*}
    \mathbb{E}^{(y)} \bigg|\sum_{\substack{n\leq x\\ P^{-}(n)>y } }f(n) \lambda(n)c_n \bigg|^2=\sum_{\substack{n\leq x \\ P^{-}(n)>y }} |\lambda(n)c_n|^2.
\end{equation*}

  As only a lower bound is required, we may restrict the range of summation to $x^{4/5} <n\leq x$, so that
\begin{equation}
\label{con}
     \mathbb{E}^{(y)} \bigg|\sum_{n\leq x}f(n)\lambda(n)\bigg|^2\geq  \sum_{\substack{x^{4/5} <n\leq x\\ P^-(n)>y } }\bigg|\lambda(n)  \sum_{\substack{m\leq x/n \\ P^+(m)\leq y }}f(m)\lambda(m)  \bigg|^2\geq  \sum_{r\leq x^{1/1000} }\bigg|  \sum_{\substack{m\leq r \\ P^+(m)\leq y }}f(m)\lambda(m) \bigg|^2  \sum_{\substack{\frac{x}{r+1}<n\leq \frac{x}{r} \\ P^-(n)>y}} |\lambda(n)|^2,
\end{equation}
where the last inequality comes by the substitution $r=\big\lfloor \frac{x}{n}\rfloor$ with $\lfloor n \rfloor = \max \{ m \in \intz : m \leq n \}$ denoting the floor of $n$. \newline

  We now deal with the sum
\begin{equation*}
     \sum_{\substack{\frac{x}{r+1}<n\leq \frac{x}{r} \\ P^-(n)>y}}   |\lambda(n)|^2
\end{equation*}
using the sieve. For our situation, we apply the lower bound part of \cite[Theorem 12.5]{FI10} for the sequence $\mathcal{A}=\{n: \frac{x}{r+1}<n\leq \frac{x}{r} \}$ with $z:=y<x^{1/10}, D=x^{6/25}$ and $X=L(1, \operatorname{sym}^2 f)\zeta^{-1}(2)(x/r-x/(r+1)) \gg x/r^2$.  By Lemma \ref{lemlambdasquaresum}, we have
\begin{align}
\label{Ad}
\begin{split}
    &  \sum_{\substack{n \in \mathcal{A} \\ d|n}} |\lambda(n)|^2=g(d)X+O(x^{3/4+\varepsilon}),
\end{split}
\end{align}
  where
\begin{align*}
\begin{split}
    &  g(d)=\frac 1{d}\prod_{\substack{ p|d \\ p^{\nu_p} \|d}} \Big ( \sum_{j=0}^{\infty} \frac{|\lambda(p^{\nu_p+j})|^2}{p^{j}} \Big ) \frac{\zeta_p(2)}{\zeta_p(1)L_p(1, \operatorname{sym}^2 f)}.
\end{split}
\end{align*}

Computation similar to that in \eqref{dfactorboundprod} reveals $g(p) =\lambda(p)^2/p+O(1/p^2)$.  So by \eqref{merten1}, the sieve dimension in our case (see \cite[(5.35)]{FI10}) is $\kappa=1$.
We now apply \cite[ Theorem 11.13]{FI10} to see that
\begin{equation*}
     \sum_{\substack{\frac{x}{r+1}<n\leq \frac{x}{r} \\ P^-(n)>y}}   |\lambda(n)|^2 \gg XV(z)(f(s)+O((\log D)^{-1/6}))+R(D,z),
\end{equation*}
  Here by \cite[(5.36)]{FI10} and $\kappa=1$ in our case,
\begin{align*}
\begin{split}
    &  V(z)=\prod_{p \leq z}(1-g(p)) \gg \frac 1{\log z} = \frac 1{\log y}.
\end{split}
\end{align*}

  Also, by \eqref{Ad} and the expression for $R(D,z)$ given on \cite[p. 207]{FI10},
\begin{align*}
\begin{split}
    &  R(D,z) \ll x^{3/4+\varepsilon}\sum_{\substack{d | \mathcal{P}(z) \\ d < D}}1.
\end{split}
\end{align*}
Since $2 \leq z \leq \sqrt{D}$ and $\kappa=1$, we can apply \cite[Lemma 12.3]{FI10}, getting
\begin{align*}
\begin{split}
    &  R(D,z) \ll x^{3/4+\varepsilon}\sum_{\substack{d | \mathcal{P}(z) \\ d < D}}1 \ll x^{3/4+\varepsilon}D \ll x^{99/100+\varepsilon}.
\end{split}
\end{align*}

   Moreover, we have $s=\log D/\log y>2$, so $f(s)\gg 1$, where $f$ is the function defined by delayed differential equations in \cite[(12.1)--(12.2)]{FI10}. We then conclude that, as $r \leq x^{1/1000}$,
\begin{equation*}
     \sum_{\substack{\frac{x}{r+1}<n\leq \frac{x}{r} \\ P^-(n)>y}}   |\lambda(n)|^2 \gg \frac{X}{\log y} \gg \frac{x}{r^2\log y}.
\end{equation*}

The above and \eqref{con} render that
\begin{equation*}
    \mathbb{E}^{(y)} \bigg|\sum_{n\leq x}f(n)\lambda(n)\bigg|^2 \gg
    \frac{x}{\log y} \sum_{r\leq x^{1/1000} } \frac{1}{r^2} \bigg|  \sum_{\substack{m\leq r \\ P^+(m)\leq y }}f(m)\lambda(m) \bigg|^2 \gg \frac{x}{\log y} \int\limits_1^{x^{1/1000}}\bigg|\sum_{\substack{m\leq t \\ P^+(m)\leq y } }f(m)\lambda(m) \bigg|^2 \frac{\dif t}{t^2}.
\end{equation*}

   We aim to apply Lemma \ref{parseval} to estimate the above integral. To that end, we first need to extend the range of integration $(1, \infty)$. For this, we apply a Rankin-type trick to see that for any $\beta>0$,
\begin{equation*}
\begin{split}
        \int\limits_1^{x^{1/1000}}\bigg|\sum_{\substack{m\leq t \\ P^+(m)<y } }f(m)\lambda(m) \bigg|^2 \frac{\dif t}{t^2} \geq  & \int\limits_1^{\infty} \bigg|\sum_{\substack{m\leq t \\ P^+(m)<y } }f(m)\lambda(m) \bigg|^2 \frac{\dif t}{t^{2+2\beta}}- \int\limits_{x^{1/1000}}^{\infty} \bigg|\sum_{\substack{m\leq t \\ P^+(m)<y } }f(m)\lambda(m) \bigg|^2 \frac{\dif t}{t^{2+2\beta} } \\
       \geq & \int\limits_1^{\infty} \bigg|\sum_{\substack{m\leq t \\ P^+(m)<y } }f(m)\lambda(m) \bigg|^2 \frac{\dif t}{t^{2+2\beta}}- x^{-\beta/1000} \int\limits_{x^{1/1000}}^{\infty} \bigg|\sum_{\substack{m\leq t \\ P^+(m)<y } }f(m)\lambda(m) \bigg|^2 \frac{\dif t}{t^{2+\beta} }\\
     \geq & \frac{1}{2\pi} \bigg(\int\limits_{-1/2}^{1/2} |F_y(1/2+\beta+it)|^2 dt-x^{-\beta/1000} \int\limits_{-\infty}^{\infty} |F_y(1/2+\beta/2+it)|^2 \frac{\dif t}{1/4+t^2} \bigg),
\end{split}
\end{equation*}
 where the last estimation above follows from Lemma \ref{parseval}. This completes the proof of the lemma.
\end{proof}

  Now, by \eqref{appending} and Lemma \ref{condition_y}, we have
\begin{equation}
\label{Elowerboundint}
\begin{split}
     \mathbb{E} \bigg|\sum_{n\leq x} & f(n)\lambda(n)\bigg|^2  \sum_{|l|\leq \log y/2} |F_{y}( \tfrac{1}{2}+il/\log y)|^{2(k-1)} \\
     & = \mathbb{E} \mathbb{E}^{(y)}\bigg|\sum_{n\leq x}f(n)\lambda(n)\bigg|^2 \sum_{|l|\leq \log y/2} |F_{y}(\tfrac{1}{2}+il/\log y)|^{2(k-1)} \\
    & \gg  \frac{x}{\log y}\sum_{|l|\leq \log y/2}\bigg[ \int\limits_{-1/2}^{1/2}\mathbb{E}|F_{y}(\tfrac{1}{2}+\beta+it)|^{2}  |F_{y}(\tfrac{1}{2}+il/\log y)|^{2(k-1)} \dif t \\
    & \hspace*{2cm} -x^{-\beta/1000} \int\limits_{-\infty}^{\infty}\mathbb{E}|F_{y}(\tfrac{1}{2}+\beta/2+it)|^{2}|F_{y}(\tfrac{1}{2}+il/\log y)|^{2(k-1)} \frac{\dif t}{1/4+t^2}\bigg]. 
\end{split}
\end{equation}
 We now set $\beta= C/\log y$ for a sufficiently large constant $C$ to be specified later. Note that for the rest of the argument, our implied constants must be independent of $C$. \newline

  We now treat the first integral in \eqref{Elowerboundint} by noting first that we have the trivial bound
\begin{equation*}
    \prod_{p\leq 200k^2}\Big| 1-\frac{\alpha_pf(p)}{p^{1/2+\beta+it}}\Big|^{-2}\Big| 1-\frac{\beta_pf(p)}{p^{1/2+\beta+it}}\Big|^{-2}\Big| 1-\frac{\alpha_pf(p)}{p^{1/2+il/\log y}}\Big|^{-2(k-1)}\Big| 1-\frac{\beta_pf(p)}{p^{1/2+il/\log y}}\Big|^{-2(k-1)}=e^{O(k^2)}.
\end{equation*}
 We then apply Lemma \ref{eulerproduct} with $\sigma_1=\beta$, $\sigma_2=0$, $t_1=t,t_2=\frac{l}{\log y}$ and $z=200k^2$ together with \eqref{alphalambdarel}, \eqref{merten1}.  This leads to
\begin{equation*}
\begin{split}
    \mathbb{E}|F_{y} & (\tfrac{1}{2}+\beta+it)|^{2}  |F_{y}(\tfrac{1}{2}+il/\log y)|^{2(k-1)} \\
    & =  \exp\bigg(\sum_{200k^2<p\leq y}\frac{\lambda^2(p)}{p^{1+2\beta }}+\frac{(k-1)^2\lambda^2(p)}{p}+\frac{2(k-1)\lambda^2(p)\cos\big((t-l/\log y)\log p\big)}{p^{1+\beta}}+O(k^2)\bigg) \\
    & =  \exp\bigg(\sum_{p\leq y}\frac{\lambda^2(p)}{p^{1+2\beta }}+\frac{(k-1)^2\lambda^2(p)}{p}+\frac{2(k-1)\lambda^2(p)\cos\big((t-l/\log y)\log p\big)}{p^{1+\beta}}+O(k^2\log\log k)\bigg).
\end{split}
\end{equation*}

  Recall the definition of $F(s)$ from \eqref{F1}. We write $F(s)$ as an Euler product and apply \eqref{lambdabound} to see that
\begin{equation}
\label{sumlogF}
\begin{split}
    \log F(1+2\beta)= &\sum_{p} \frac{\lambda^2(p)}{p^{1+2\beta} }+O(1).  
\end{split}
\end{equation}  
Using $\sum_{p>y} p^{-1-1/\log y} \ll 1$ and $\log \zeta(1+s)=-\log s +O(1)$ for $s\ll 1$, we see from the above and \eqref{Frel1} that, for $C\geq 1/2$,
\begin{equation}
\label{sum1}
\begin{split}
    \sum_{p\leq y} \frac{\lambda^2(p)}{p^{1+2\beta} } \geq &\sum_{p} \frac{\lambda^2(p)}{p^{1+2\beta} } -\sum_{p\geq y} \frac{\lambda^2(p)}{p^{1+1/\log y}}  = \log F(1+2\beta)+O(1) \\
   =&\log \zeta(1+2\beta)+\log L(1+2\beta, \operatorname{sym}^2 f)-\log \zeta(2+4\beta)+O(1)  
    = \log\log y-\log C+O(1).
\end{split}
\end{equation}

  Also, by \eqref{merten1},
\begin{equation}
\label{sum2}
\begin{split}
    \sum_{p\leq y} \frac{(k-1)^2\lambda^2(p)}{p } \geq &  (k-1)^2\log\log y +O(k^2).
\end{split}
\end{equation}
Similarly, using $\Re \log \zeta(1+s)=- \log |s|+O(1)$ for $s\ll 1$, we see that when $\big|t-l/\log y\big|\leq 1/\log y$ and $C\geq 3$, 
\begin{equation}
\label{sum3}
\begin{split}
   \sum_{p\leq y}  \frac{\lambda^2(p)\cos\big((t-l/\log y)\log p\big)}{p^{1+\beta}} &\geq \Re \log \zeta\big(1+\beta+ (t-l/\log y)i\big)+O(1) \\
   &= -\log |\beta +(t-l/\log y)i | +O(1) \\
   &=\log\log y-\frac{1}{2}\log\big(C^2+ (t\log y-l)^2\big)+O(1) \\
   &=\log\log y-\log C+O(1).
\end{split}
\end{equation}

  It follows from \eqref{sum1}--\eqref{sum3} that
\begin{align}
\label{positive_expected}
\begin{split}
   \int\limits_{-1/2}^{1/2} & \mathbb{E}  |F_y( \tfrac{1}{2}+\beta+it)|^{2}  |F_y(\tfrac{1}{2}+il/\log y)|^{2(k-1)} \dif t \\
& \geq \int\limits_{-1/2}^{1/2} \mathbb{E}|F_y( \tfrac{1}{2}+\beta+it)|^{2}  |F_y( \tfrac{1}{2}+il/\log y)|^{2(k-1)} \mathbf{1} \Big( \big|t-\frac{l}{\log y}\big|\leq \frac{1}{\log y} \Big) \dif t \geq e^{O(k^2\log \log k)} (\log y)^{k^2-1}C^{1-2k}.
\end{split}
\end{align}

For the second integral in \eqref{Elowerboundint}, we assume that $l=0$, as the other cases can be treated similarly. For any $t\in \mathbb{R}$, we apply Lemma \ref{eulerproduct} and arguing as above to see that
\begin{align}
\label{SecondEupperbound}
\begin{split}
   & \mathbb{E}|F_y(\tfrac{1}{2}+\tfrac{\beta}{2}+it)|^{2}|F_y(\tfrac{1}{2})|^{2(k-1)} \leq  (\log y)^{(k-1)^2+1}\exp\bigg(2(k-1) \sum_{p\leq y}\frac{\lambda^2(p)\cos(t\log p)}{p^{1+\beta/2}}+O(k^2\log \log k)\bigg) .
\end{split}
\end{align}

By \eqref{sumlambdapsquare} and Lemma \ref{RS3} and keeping in mind that the last estimation given in \eqref{mertenstypesympower} equals $O(1)$ when $|\alpha| \leq e^e$, we see that
\begin{equation*}
\begin{split}
    \int\limits_{-\infty}^{\infty} \exp & \bigg(2(k-1) \sum_{p\leq y}\frac{\lambda^2(p)\cos(t\log p)}{p^{1+\beta/2}}\bigg)\frac{\dif t}{t^2+1/4} = \int\limits_{-\infty}^{\infty} \exp\bigg(2(k-1) \sum_{p\leq y}\frac{(\lambda(p^2)+1)\cos(t\log p)}{p^{1+\beta/2}}\bigg)\frac{\dif t}{t^2+1/4} \\
    \leq &e^{O(k)}\bigg[\int\limits_0^{1/\log y} (\log y)^{2(k-1)} \dif t+\int\limits_{1/\log y}^{10} t^{-2(k-1)} \dif t+\int\limits_{10}^{\infty} \frac{(\log t)^{8(k-1)} }{t^2} \dif t \bigg] \leq e^{O(k\log k)}\big( (\log y)^{2k-3}+1\big).
\end{split}
\end{equation*}

 As $k\geq 2$, we have $(\log y)^{2k-3} \gg 1$ so that the above and \eqref{SecondEupperbound} give
\begin{align}
\label{secondintegral}
\begin{split}
    \int\limits_{-\infty}^{\infty}\mathbb{E}|F_y(\tfrac{1}{2}+ \tfrac{\beta}{2}+it)|^{2}|F_y(\tfrac{1}{2}+il/\log y)|^{2(k-1)} \frac{\dif t}{1/4+t^2}\leq e^{O(k^2\log \log k)}  (\log y)^{k^2-1}.
\end{split}
\end{align}

   We deduce from \eqref{positive_expected} and \eqref{secondintegral} that
\begin{equation*}
\begin{split}
    \int\limits_{-1/2}^{1/2}  \mathbb{E}|F_y(\tfrac{1}{2}+\beta+it)|^2 & |F_y( \tfrac{1}{2}+  il/\log y)|^{2(k-1)} \dif t -x^{-\beta/1000}  \int_{-\infty}^{\infty}\mathbb{E}|F_y(\tfrac{1}{2}+\tfrac{\beta}{2}+it)|^2|F_y(\tfrac{1}{2}+il/\log y)|^{2(k-1)} \frac{\dif t}{1/4+t^2} \\
    \geq & (\log y)^{k^2-1}(C^{1-2k}e^{O(k^2\log \log k)}-e^{-C/1000} e^{O(k^2\log\log k)} ).
\end{split}
\end{equation*}
  We now choose $C$ to be a large multiple of $k^2\log\log k$.  The assertion of Proposition \ref{p3} follows from the above.

\subsection{Proof of Proposition \ref{p4}}

We denote $I_m=(y_{m-1}, y_m]$ and $\mathcal{P}_m=\{p\leq y: \,  p\not\in I_m\}$ for $1\leq m\leq M$. We first establish the following result. 
\begin{lemma}
\label{hardestlemma}
   For any $t\in \mathbb{R}$ and $j\in \mathbb{N}$, we have
    \begin{equation}
\label{E1}
    \begin{split}
       \mathbb{E}  \bigg|\sum_{n\leq x} & f(n) \lambda(n)\bigg|^2 \bigg|\exp\bigg( (k-1)\sum_{p\in \mathcal{P}_m} \frac{(\alpha_p+\beta_p)f(p)}{p^{1/2+it}}+\frac{(\alpha^2_p+\beta^2_p)f(p)^2}{2p^{1+2it} }\bigg)  \bigg|^2 \bigg| \sum_{p\in I_m} \frac{(\alpha_p+\beta_p)f(p)}{p^{1/2+it}}+\frac{(\alpha^2_p+\beta^2_p)f(p)^2}{2p^{1+2it} }\bigg|^{2j} \\
        \leq & e^{O(k^4)}x(\log y)^{k^2-2}j! \bigg(4\sum_{p\in I_m} \frac{2}{p}+\frac{3}{p^2}\bigg)^j \prod_{p\in I_m}\bigg(1-\frac{\lambda^2(p)}{p}\bigg)^{-2}\frac{\log x}{\log y}.
    \end{split}
    \end{equation}
\end{lemma}
\begin{proof}
We denote by $\mathbb{E}^{(\mathcal{P}_m)}$ the conditional expectation with respect to primes in $\mathcal{P}_m$. The tower rule $\mathbb{E}=\mathbb{E}\mathbb{E}^{(\mathcal{P}_m)}$ leads to that the expectation in \eqref{E1} equals to
    \begin{equation}
\label{E1cond}
    \begin{split}
        \mathbb{E}\mathbb{E}^{(\mathcal{P}_m)} & \bigg|\sum_{n\leq x}f(n)\lambda(n)\bigg|^2 \bigg|\exp\bigg( (k-1)\sum_{p\in \mathcal{P}_m} \frac{(\alpha_p+\beta_p)f(p)}{p^{1/2+it}}+\frac{(\alpha^2_p+\beta^2_p)f(p)^2}{2p^{1+2it} }\bigg)  \bigg|^2 \\
        & \hspace*{4cm} \times \bigg| \sum_{p\in I_m} \frac{(\alpha_p+\beta_p)f(p)}{p^{1/2+it}}+\frac{(\alpha^2_p+\beta^2_p)f(p)^2}{2p^{1+2it} }\bigg|^{2j} \\
        =&\mathbb{E}\bigg|\exp\bigg( (k-1)\sum_{p\in \mathcal{P}_m} \frac{(\alpha_p+\beta_p)f(p)}{p^{1/2+it}}+\frac{(\alpha^2_p+\beta^2_p)f(p)^2}{2p^{1+2it} }\bigg)  \bigg|^2 \\
        & \hspace*{3cm} \times \mathbb{E}^{(\mathcal{P}_m)}  \bigg|\sum_{n\leq x}f(n)\lambda(n)\bigg|^2 \bigg| \sum_{p\in I_m} \frac{(\alpha_p+\beta_p)f(p)}{p^{1/2+it}}+\frac{(\alpha^2_p+\beta^2_p)f(p)^2}{2p^{1+2it} }\bigg|^{2j}.
    \end{split}
    \end{equation}
Applying the same trick introduced at the beginning of the proof of Lemma \ref{condition_y}, we arrive at
\begin{equation*}
\begin{split}
\mathbb{E}^{(\mathcal{P}_m)} &  \bigg|\sum_{n\leq x}f(n)\lambda(n)\bigg|^2 \bigg| \sum_{p\in I_m} \frac{(\alpha_p+\beta_p)f(p)}{p^{1/2+it}}+\frac{(\alpha^2_p+\beta^2_p)f(p)^2}{2p^{1+2it}} \bigg|^{2j}\\
& = \mathbb{E}^{(\mathcal{P}_m)}\bigg|\sum_{\substack{n\leq x \\p|n \implies p\not\in \mathcal{P}_m}}f(n)\lambda(n)\sum_{\substack{l\leq x/n \\p|l \implies p\in \mathcal{P}_m}} f(l)\lambda(l)\bigg|^2\bigg| \sum_{p\in I_m} \frac{(\alpha_p+\beta_p)f(p)}{p^{1/2+it}}+\frac{(\alpha^2_p+\beta^2_p)f(p)^2}{2p^{1+2it} }\bigg|^{2j}. 
\end{split}
\end{equation*}
  We write
\begin{equation*}
\Tilde{d}(n)=\sum_{d|n} \mathbf{1}(p|d \implies p\in I_m) \; \mbox{and} \quad    c_n=\sum_{\substack{l\leq x/n \\p|l \implies p\in \mathcal{P}_m}} f(l)\lambda(l). 
\end{equation*}
Because of the condition on primes in $\mathcal{P}_m$, we may regard $c_n$ as a fixed quantity.  Lemma \ref{evenmoment}, together with the bounds $|\alpha_p+\beta_p| \leq 2, |\alpha^2_p+\beta^2_p| \leq 2$ which follow from \eqref{alpha}, gives that
\begin{align*}
   \mathbb{E}^{(\mathcal{P}_m)}\bigg| & \sum_{\substack{n\leq x \\p|n \implies p\not\in \mathcal{P}_m}}f(n)  \lambda(n)c_n\bigg|^2\bigg|\sum_{p\in I_m} \frac{(\alpha_p+\beta_p)f(p)}{p^{1/2+it}}+\frac{(\alpha^2_p+\beta^2_p)f(p)^2}{2p^{1+2it} }\bigg|^{2j} \\
& \ll  \bigg(\sum_{\substack{n\leq x \\p|n \implies p\not\in \mathcal{P}_m}}\Tilde{d}(n) |\lambda(n)c_n|^2\bigg) j! \bigg(4\sum_{p\in I_m}\frac{2}{p}+\frac{3}{p^2}\bigg)^j.
\end{align*}
  We deduce from \eqref{E1cond} and the above that 
\begin{align}
\label{E11}
\begin{split}
    \mathbb{E} \bigg|\sum_{n\leq x} & f(n)\lambda(n)\bigg|^2 \bigg|\exp\bigg( (k-1)\sum_{p\in \mathcal{P}_m} \frac{(\alpha_p+\beta_p)f(p)}{p^{1/2+it}}+\frac{(\alpha^2_p+\beta^2_p)f(p)^2}{2p^{1+2it} }\bigg)  \bigg|^2 \\
    & \hspace*{3cm} \times \bigg| \sum_{p\in I_m} \frac{(\alpha_p+\beta_p)f(p)}{p^{1/2+it}}+\frac{(\alpha^2_p+\beta^2_p)f(p)^2}{2p^{1+2it} }\bigg|^{2j} \\
& \ll j! \bigg(4\sum_{p\in I_m}\frac{2}{p}+\frac{3}{p^2}\bigg)^j\sum_{\substack{n\leq x \\p|n \implies p\not\in \mathcal{P}_m}}\Tilde{d}(n)|\lambda(n)|^2\mathbb{E} \bigg|\sum_{\substack{l\leq x/n \\p|l \implies p\in \mathcal{P}_m}} f(l)\lambda(l)\bigg|^2 \\
& \hspace*{3cm} \times \bigg|\exp\bigg( (k-1)\sum_{p\in \mathcal{P}_m} \frac{(\alpha_p+\beta_p)f(p)}{p^{1/2+it}}+\frac{(\alpha^2_p+\beta^2_p)f(p)^2}{2p^{1+2it} }\bigg)  \bigg|^2.
\end{split}
\end{align}

   Note that trivially,  
\begin{equation}
\label{e1}
    \bigg|\exp\bigg( (k-1)\sum_{\substack{p\in \mathcal{P}_m \\ p\leq 10k^2}} \frac{(\alpha_p+\beta_p)f(p)}{p^{1/2+it}}+\frac{(\alpha^2_p+\beta^2_p)f(p)^2}{2p^{1+2it}}\bigg)  \bigg|^2=e^{O(k^2)}.
\end{equation}

   For any prime $p\geq 10k^2$, we apply the Taylor series expansion to see that
\begin{align}
\label{e2}
\begin{split}
          \exp & \bigg( (k-1)\Big(\frac{(\alpha_p+\beta_p)f(p)}{p^{1/2+it}}+\frac{(\alpha^2_p+\beta^2_p)f(p)^2}{2p^{1+2it}}\Big)\bigg) \\
        & = 1+(k-1)\frac{(\alpha_p+\beta_p)f(p)}{p^{1/2+it}}+\frac 12((k-1)^2(\alpha_p+\beta_p)^2+(k-1)(\alpha^2_p+\beta^2_p))\frac{f^2(p)}{p^{1+2it}}+O\Big(\frac{k^3}{p^{3/2} }\Big) \\
        & = 1+(k-1)\frac{(\alpha_p+\beta_p)f(p)}{p^{1/2+it}}+\Big( \frac{k(k-1)(\alpha^2_p+\beta^2_p)}{2}+(k-1)^2 \Big)\frac{f^2(p)}{p^{1+2it}}+O\Big(\frac{k^3}{p^{3/2} }\Big).
\end{split}
\end{align}
For $p\geq 10k^2$, 
\begin{equation*}
   \bigg| 1+(k-1)\frac{(\alpha_p+\beta_p)f(p)}{p^{1/2+it}}+\Big(\frac{k(k-1)(\alpha^2_p+\beta^2_p)}{2}+(k-1)^2\Big)\frac{f^2(p)}{p^{1+2it}}\bigg|\geq \frac{1}{2}.
\end{equation*}
Thus, for $p\geq 10k^2$, 
\begin{align}
\label{e3}
\begin{split}
   & \bigg| 1+(k-1)\frac{(\alpha_p+\beta_p)f(p)}{p^{1/2+it}}+\Big(\frac{k(k-1)(\alpha^2_p+\beta^2_p)}{2}+(k-1)^2 \Big)\frac{f^2(p)}{p^{1+2it}}+O\Big(\frac{k^3}{p^{3/2} }\Big) \bigg|^2 \\
   \leq & \bigg| 1+(k-1)\frac{(\alpha_p+\beta_p)f(p)}{p^{1/2+it}}+\Big(\frac{k(k-1)(\alpha^2_p+\beta^2_p)}{2}+(k-1)^2 \Big)\frac{f^2(p)}{p^{1+2it}}\bigg|^2\Big( 1+O\Big(\frac{k^3}{p^{3/2}}\Big)\Big),
\end{split}
\end{align}
  We deduce from \eqref{e1}--\eqref{e3} that
\begin{align*}
\begin{split}
     \bigg|\exp & \bigg( (k-1)\sum_{p\in \mathcal{P}_m} \frac{(\alpha_p+\beta_p)f(p)}{p^{1/2+it}}+\frac{(\alpha^2_p+\beta^2_p)f(p)^2}{2p^{1+2it} }\bigg)  \bigg|^2 \\
  &  \leq   e^{O(k^2)} \prod_{\substack{p\in \mathcal{P}_m \\ p\geq 10k^2}}\bigg| 1+(k-1)\frac{(\alpha_p+\beta_p)f(p)}{p^{1/2+it}}+\Big(\frac{k(k-1)(\alpha^2_p+\beta^2_p)}{2}+(k-1)^2\Big)\frac{f^2(p)}{p^{1+2it}}\bigg|^2.
\end{split}
\end{align*}

 We now write
\begin{equation*}
    \prod_{\substack{p\in \mathcal{P}_m \\ p\geq 10k^2 }}\bigg ( 1+(k-1)\frac{(\alpha_p+\beta_p)f(p)}{p^{1/2+it}}+\Big(\frac{k(k-1)(\alpha^2_p+\beta^2_p)}{2}+(k-1)^2\Big)\frac{f^2(p)}{p^{1+2it}}\bigg)=\sum_{v}\frac {h(v)f(v)}{v^{1/2+it}}, 
\end{equation*}
  where $h$ is a multiplicative function such that $h(1)=1, h(p)=(k-1)(\alpha_p+\beta_p)=(k-1)\lambda(p)$ and $h(p^2)=\tfrac12 k(k-1)(\alpha^2_p+\beta^2_p)+(k-1)^2$ for $p\in \mathcal{P}_m,  p\geq 10k^2$ and $h=0$ otherwise. \newline

 Observe that the orthogonality relation $\mathbb{E} f(n)\overline{f(m)}=\mathbf{1}(n=m)$ implies $\mathbb{E}\big| \sum_{n\leq N} a_nf(n)\big|^2=\sum_{n\leq N} |a_n|^2$ for any set of complex coefficients $(a_n)_{n\leq N}$.  Hence,
\begin{align}
\label{square}
\begin{split}
    \mathbb{E} & \bigg|\sum_{\substack{l\leq x/n \\p|l \implies p\in \mathcal{P}_m}} f(l)\lambda(l)\bigg|^2\prod_{\substack{p\in \mathcal{P}_m \\ p\geq 10k^2}}\bigg| 1+(k-1)\frac{(\alpha_p+\beta_p)f(p)}{p^{1/2+it}}+ \Big(\frac{k(k-1)(\alpha^2_p+\beta^2_p)}{2}+(k-1)^2 \Big)\frac{f^2(p)}{p^{1+2it}}\bigg|^2 \\
& =  \mathbb{E}\bigg|\sum_{\substack{l\leq x/n \\p|l \implies p\in \mathcal{P}_m}} f(l)\lambda(l)\bigg|^2\bigg| \sum_{v}\frac {h(v)f(v)}{v^{1/2+it}}\bigg|^2  = \sum_{u}\bigg|\sum_{\substack{u=lv\\ l\leq x/n\\ p|l \implies p\in \mathcal{P}_m }}\frac{\lambda(l)h(v)}{v^{1/2+it} }\bigg|^2 \\
& \leq  \sum_{u}\bigg(\sum_{\substack{u=lv\\ l\leq x/n}}\frac{|\lambda(l)h(v)|}{v^{1/2} }\bigg)^2 =  \sum_{\substack{l_1v_1=l_2v_2 \\ l_1,l_2\leq x/n }}\frac{|\lambda(l_1)h(v_1)\lambda(l_2)h(v_2)|}{(v_1v_2)^{1/2} }  = \sum_{v_1,v_2}\frac{|h(v_1)h(v_2)|}{(v_1v_2)^{1/2} }\sum_{\substack{l_1,l_2\leq x/n \\ l_1/l_2=v_2/v_1 } }|\lambda(l_1)\lambda(l_2)| \\
& = \sum_{\substack{d,w_1,w_2 \\ (w_1, w_2)=1}}\frac{|h(dw_1)h(dw_2)|}{d(w_1w_2)^{1/2}}\sum_{\substack{m \leq x/(n\max(w_1,w_2)) }} |\lambda(w_2m)\lambda(w_1m)| ,
\end{split}
\end{align}
  where the last equality above follows by setting $d=(v_1,v_2)$, $v_1=dw_1$, $v_2=dw_2$ so that $(w_1, w_2)=1$ and by noting that the condition $l_1/l_2=v_2/v_1$ implies that $l_1=w_2m$, $l_2=w_1m$ for some positive integer $m$. Note that the in the computation \eqref{square}, we simply use the triangle inequality.  We remark here that by doing so, we still expect to get a reasonable upper bound since, from the arguments below, the most important contribution comes from products of roughly the form $\prod_p(1+\frac {\lambda^2(p)}{p})$ with $p$ belonging to certain sets. As $\lambda(p)$ is real and hence $\lambda(p)^2=|\lambda(p)|^2$, we do not lose much here. \newline

Now Lemma \ref{lemlambdasquaresum1} gives
\begin{align*}
\begin{split}
    & \sum_{\substack{l_1,l_2\leq x/n \\ l_1/l_2=v_2/v_1 } }|\lambda(l_1)\lambda(l_2)| = \sum_{\substack{m \leq x/(n\max(w_1,w_2)) }} |\lambda(w_2m)\lambda(w_1m)| \leq  \frac{C_1x}{n\max(w_1,w_2)}(P_1(w_1w_2)+P_2(w_1w_2)). 
\end{split}
\end{align*}

Thus the quantities in \eqref{square} are 
\begin{align*}
\begin{split}
    \leq & C_1\sum_{\substack{d,w_1,w_2 \\ (w_1,w_2)=1}}\frac{|h(dw_1)h(dw_2)|}{d(w_1w_2)^{1/2} }\frac{x}{n\max(w_1,w_2)}(P_1(w_1w_2)+P_2(w_1w_2)). 
\end{split}
\end{align*}
   
    Without any loss of generality, we may assume that $w_1\geq w_2$.  As the treatments are similar, we only consider the estimation involving with $P_1(w_1w_2)$ in what follows. Thus we see that the quantities in \eqref{square} are 
\begin{align*}
\begin{split}
    \leq C_1\sum_{\substack{d,w_1,w_2 \\ (w_1,w_2)=1, w_1 \geq w_2}}\frac{|h(dw_1)h(dw_2)|}{d(w_1w_2)^{1/2} }\frac{x}{nw_1}P_1(w_1w_2) \leq C_1\sum_{\substack{w_1,w_2 \\ (w_1,w_2)=1, w_1 \geq w_2}}\frac{P_1(w_1w_2)}{(w_1w_2)^{1/2} }\frac{x}{n w_1} \sum_d\frac{|h(dw_1)h(dw_2)|}{d}. 
\end{split}
\end{align*}

By \eqref{merten1}, 
\begin{equation}
\label{estsumd}
    \sum_{d}\frac{|h(dw_1)h(dw_2)|}{d}\leq \prod_{p\leq y }\bigg(1+\frac{(k-1)^2\lambda^2(p)}{p}+\frac{h^2(p^2)}{p^2}\bigg)P_3(w_1w_2)\leq e^{O(k^4)}(\log y )^{(k-1)^2}P_3(w_1w_2),
\end{equation}
   where, for any integer $c$, 
\begin{align}
\label{P3}
\begin{split}
   P_3(c) = 
   \prod_{\substack {p\| c}}\bigg(|h(p)|+\frac{|h(p^2)h(p)|}{p}\bigg)\prod_{\substack {p^2 \| c}}|h(p^2)|.
\end{split}
\end{align}

   We thus deduce from the above and the fact that $(w_1,w_2)=1$ that the quantities in \eqref{square} are 
\begin{align*}
\begin{split}
   \leq C_1e^{O(k^4)}\sum_{\substack{w_1,w_2 \\ (w_1,w_2)=1, w_1 \geq w_2}} & \frac{1}{(w_1w_2)^{1/2} }\frac{x}{nw_1}P_1(w_1w_2)P_3(w_1w_2) \\
  \leq & C_1e^{O(k^4)} (\log y )^{(k-1)^2} \sum_{\substack{w_1,w_2 \\ (w_1,w_2)=1, w_1 \geq w_2}}\frac{1}{(w_1w_2)^{1/2} }\frac{x}{nw_1}P_1(w_1)P_1(w_2)P_3(w_1)P_3(w_2). 
\end{split}
\end{align*}

   Note that $P_1(n)P_3(n)$ is a non-negative multiplicative function of $n$. Here as usual, the empty product equals to $1$. We apply \eqref{alpha}, \eqref{merten}, \eqref{P12} and \eqref{P3} to see that
\begin{align}
\label{psumbound}
\begin{split}
   \sum_{p \leq x}P_1(p)P_3(p)\log p & \leq O(k^2)(\sum_{p \leq x}|\lambda(p)|^2\log p+O(1)) \\
   & = O(k^2)\sum_{p \leq x}\lambda(p)^2\log p+O(k^2) \leq O(k^2)\sum_{p \leq x}\log p \leq  O(k^2)x. 
\end{split}
\end{align}  
   Similarly,
\begin{align}
\label{phigherpowersumbound}
\begin{split}
   \sum_{\substack{p^l \leq x \\ l \geq 2}}\frac {P_1(p^l)P_3(p^l)l\log p}{p^l} \leq O(k^2)\sum_{\substack{p^l \leq x \\ l \geq 2}}\frac {|\lambda(p^l)\lambda(p)| l \log p}{p^l} \ll  k^2. 
\end{split}
\end{align}  

  The estimations given in \eqref{psumbound} and \eqref{phigherpowersumbound} allow us to apply Theorem 2.14 and Corollary 2.15 of \cite{MVa1} with  $A=O(k^2)$. Hence, for any $1 \leq u \leq w_1$, upon using $|\lambda(p^{\nu})| \leq d(p^{\nu}) \leq \nu+1$ for all positive integer $\nu$, 
\begin{equation*}
\begin{split}
    \sum_{w_2\leq u} P_1(w_2)P_3(w_2)\leq & O(k^2)\frac{u}{\log u} \prod_{\substack{ p\leq u \\p\in \mathcal{P}_m} }\bigg(1+\frac{P_1(p)P_3(p)}{p}+\frac{P_1(p^2)P_3(p^2)}{p^2}+\cdots \bigg) \\
    \leq & O(k^2)\frac{u}{\log u} \prod_{\substack{ p\leq u \\p\leq y }}\bigg(1+\frac{(k-1)\lambda^2(p)}{p}+O(\frac{k^2}{p^2})\bigg) \\
    \leq & e^{O(k^2)}u(\log y)^{k-2}, 
\end{split}
\end{equation*}
   where the last bound above follows from the inequality $1+x \leq e^x$ for all $x \in \rear$ together with \eqref{merten1}.  We deduce from the above and partial summation that
\begin{equation}
\label{estsumw2}
    \sum_{w_2\leq w_1} \frac{P_1(w_2)P_3(w_2)}{w_2^{1/2}}\leq e^{O(k^2)}w_1^{1/2}(\log y)^{k-2}.
\end{equation}
  We note furthermore that 
\begin{equation}
\label{estsumw1}
    \sum_{w_1}\frac{P_1(w_1)P_3(w_1)}{w_1}\leq \prod_{p\leq y}\bigg(1+\frac{(k-1)\lambda^2(p)}{p}+O(\frac{k^2}{p^2})\bigg)\leq e^{O(k^2)}(\log y)^{k-1}. 
\end{equation}
Thus from \eqref{estsumd}, \eqref{estsumw2} and \eqref{estsumw1} that the quantities in \eqref{square} are at most
\begin{equation*}
   \frac{x}{n} e^{O(k^4)}(\log y)^{k^2-2}.
\end{equation*}
  It follows from \eqref{e1}, \eqref{square} and the above that
\begin{align}
\label{E12}
\begin{split}
  \sum_{\substack{n\leq x \\p|n \implies p\not\in \mathcal{P}_m}}\Tilde{d}(n) & |\lambda(n)|^2 \mathbb{E} \bigg|\sum_{\substack{l\leq x/n \\p|l \implies p\in \mathcal{P}_m}} f(l)\lambda(l)\bigg|^2 \bigg|\exp\bigg( (k-1)\sum_{p\in \mathcal{P}_m} \frac{(\alpha_p+\beta_p)f(p)}{p^{1/2+it}}+\frac{(\alpha^2_p+\beta^2_p)f(p)^2}{2p^{1+2it} }\bigg)  \bigg|^2\\
& \leq e^{O(k^4)}x(\log y)^{k^2-2}\sum_{\substack{n\leq x \\p|n \implies p\not\in \mathcal{P}_m}}\frac{\Tilde{d}(n)|\lambda(n)|^2}{n}.
\end{split}
\end{align}  

  Recalling that $\Tilde{d}(n)=\sum_{d|n} \mathbf{1}(p|d \implies p\in I_m)$ and noting that the condition $p\leq x, p\not \in \mathcal{P}_m$ implies $p\in I_m$ or $p\in [y,x]$, we see that
\begin{equation*}
    \sum_{\substack{n\leq x \\p|n \implies p\not\in \mathcal{P}_m}}\frac{\Tilde{d}(n)|\lambda(n)|^2}{n}=\sum_{\substack{d\\ p|d\implies p\in I_m} }\frac{1}{d}\sum_{\substack{n\leq x/d \\ p|n\implies p\not\in \mathcal{P}_m }}\frac{|\lambda(dn)|^2}{n} \\
\leq \sum_{\substack{d, n\\ p|d\implies p\in I_m \\ p|n\implies p\not\in \mathcal{P}_m }}\frac{|\lambda(dn)|^2}{dn}.
\end{equation*}

  Note that the last expression above is jointly multiplicative in both $n$ and $d$.  So we can express the double sum in terms of an Euler product. We then deduce that
\begin{align*}
\begin{split}
    \sum_{\substack{n\leq x \\p|n \implies p\not\in \mathcal{P}_m}}\frac{\Tilde{d}(n)|\lambda(n)|^2}{n}
 \ll & \prod_{p\in I_m}\Big(\sum_{n_1, n_2 \geq 0}\frac {|\lambda(p^{n_1+n_2})|^2}{p^{n_1+n_2}}\Big)\prod_{y\leq p\leq x}\Big(\sum_{n_3 \geq 0}\frac {|\lambda(p^{n_3})|^2}{p^{n_3}}\Big) \\
\ll & \prod_{p\in I_m}\Big(1+\frac {2|\lambda(p)|^2}{p}+O\Big(\frac 1{p^2}\Big)\Big)\prod_{y\leq p\leq x}\Big(1+\frac {|\lambda(p)|^2}{p}+O\Big(\frac 1{p^2}\Big)\Big) \\
\ll & \prod_{p\in I_m}\Big(1+\frac {2|\lambda(p)|^2}{p}\Big)\prod_{y\leq p\leq x}\Big(1+\frac {|\lambda(p)|^2}{p}\Big) \ll \prod_{p\in I_m}\Big(1-\frac{|\lambda(p)|^2}{p}\Big)^{-2}\prod_{y\leq p\leq x}\Big(1-\frac{|\lambda(p)|^2}{p}\Big)^{-1}.
\end{split}
\end{align*}  
  As $\prod_{y\leq p\leq x}\big(1-|\lambda(p)|^2p^{-1}\big)^{-1}\ll \log x/\log y$ by \eqref{merten1}, the assertion of Lemma \ref{hardestlemma} now follows from this, \eqref{E11} and \eqref{E12}. 
\end{proof}

 Now, the proof of Proposition \ref{p4} follows by a straightforward modification of the proof of \cite[Proposition 4.2]{Szab24}, upon using Lemma \ref{hardestlemma}. For any $1\leq m \leq M$ and $|l|\leq \frac{\log y}{2}$, set
\begin{align}
\label{KAdef}
\begin{split}
    K_{m,l}(f)= & \bigg|\sum_{n\leq x}f(n)\lambda(n)\bigg|^2\exp\Big(2(k-1)\sum_{\substack{m'=1\\ m'\neq m} }^M\Re D_{m',l}(f)\Big) \\
=& \bigg|\sum_{n\leq x}f(n)\lambda(n)\bigg|^2 \bigg|\exp\bigg( (k-1)\sum_{p\in \mathcal{P}_m} \frac{(\alpha_p+\beta_p)f(p)}{p^{1/2+il/\log y}}+\frac{(\alpha^2_p+\beta^2_p)f(p)^2}{2p^{1+2il/\log y} }\bigg)  \bigg|^2,  \quad \mbox{and} \\
  A_m=& 4\sum_{p\in I_m}\Big(\frac{2}{p}+\frac{3}{p^2}\Big). 
\end{split}
\end{align}  

  By Lemma \ref{hardestlemma} and arguing as in the proof of \cite[Proposition 4.2]{Szab24}, we get that 
\begin{equation*}
\begin{split}
    &\mathbb{E} K_{m,l}(f) |\text{Err}_{m,l}(f)|  \leq  e^{O(k^4)}x(\log y)^{k^2-2} \frac{\log x}{\log y}\prod_{p\in I_m}\bigg(1-\frac{\lambda^2(p)}{p}\bigg)^{-2} \sum_{\substack{j_1,j_2\geq 0\\ \max(j_1,j_2)>J_m }} \frac{(k-1)^{j_1+j_2}}{j_1!j_2!}\Big\lceil \frac{j_1+j_2}{2} \Big\rceil !  A_m^{\frac{j_1+j_2}{2}}. \\
\end{split}
\end{equation*}
Now \eqref{merten1} implies that $\prod_{p\in I_m}\big(1-\lambda^2(p)p^{-1}\big)^{-2}\ll (\log y)^2\leq e^{J_1}$ for $m=1$. Similarly, $\prod_{p\in I_m}\big(1-\lambda^2(p)p^{-1}\big)^{-2}\leq 10^4 \leq e^{J_M}\leq e^{J_m}$ for $m\geq 2$. Therefore, it remains to show that the inner sum above is at most $e^{-2J_m}$. Without loss of generality, we may assume that $j_1\geq j_2$. Then as shown in the proof of \cite[Proposition 4.2]{Szab24} that the inner sum is at most $e^{-2J_m}$ provided that we have
\begin{equation}
\label{estAJ}
    10^4(k-1)^2A_m\leq J_m.
\end{equation}
 Note that $A_m=4P_m$, where $P_m$ is defined in the proof of \cite[Lemma 9]{Szab24}. Thus the estimates for $P_m$ given in the proof of \cite[Lemma 9]{Szab24} yield
\begin{align}
\label{Abound}
\begin{split}
    A_m \leq \begin{cases}
        40, & \text{  if  } m\geq 2, \\
        12\log \log y, & \text{  if  } m=1. 
    \end{cases}
\end{split}
\end{align}  
  We recall that $J_1=(\log \log y)^{3/2}$ and note that $J_m\geq J_M\geq 10^{10}k^2$. It follows that the estimation given in \eqref{estAJ} is valid for all $m$. 
This thus completes the proof of Proposition \ref{p4}.

\section{Proof of Proposition \ref{p2}}

  As shown in the proof of \cite[Proposition 3.2]{Szab24}, we have for $k\geq 2$, 
\begin{equation*}
\begin{split}
      R(\chi)^{\frac{k}{k-1}}
      \leq & \sum_{|l_1|,|l_2|\leq (\log y) /2} \prod_{m=1}^M R_{m,l_1}(\chi)R_{m,l_2}(\chi)^{1/(k-1)}. \\
\end{split}
\end{equation*}

   Suppose that for fixed $l_1, l_2$, we have
\begin{equation}
\label{sumRbound}
    \frac{1}{\varphi(q)}\sum_{\chi\,\text{mod}\, q}\prod_{m=1}^M R_{m,l_1}(\chi)R_{m,l_2}(\chi)^{1/(k-1)}\ll \frac{(\log y)^{k^2} }{|l_1-l_2|^{2(k-1)} +1}. 
\end{equation}
  Then summing over $l_1$ and $l_2$ leads to the estimation
\begin{equation*}
        \frac{1}{\varphi(q)}\sum_{\chi\,\text{mod}\, q} R(\chi)^{k/(k-1)}\ll_k (\log y)^{k^2+1}.
\end{equation*}
  This above estimation now implies the assertion given in Proposition \ref{p2} is valid. \newline

  Thus, it remains to establish \eqref{sumRbound}. For this, we partition $[0,\infty)$ into intervals $(I^{(m)}_n)_{n\geq 0}$ for any $1\leq m\leq M$, where $I^{(m)}_0=[0, \frac{J_m}{100k}] $. We also define the dyadic interval $I^{(m)}_n=\frac{J_m}{100k}\cdot [2^{n-1}, 2^n]$ for any $n\geq 1$.  Set
\begin{equation*}
    \mathcal{X}(n_1,\ldots, n_M)=\{ \chi\in \mathcal{X}_q : |\Re D_{m,l_2}(\chi)|\in I^{(m)}_{n_m} \; \mbox{for all}  \; 1\leq m\leq M \}.
\end{equation*}

  We now fix non-negative integers $n_1,\ldots n_M$ and consider $\chi\in \mathcal{X}(n_1,\ldots, n_M)$.  We write $W_m=\inf I^{(m)}_{n_m}$ for each $1 \leq m \leq M$.
Let $a_m= 2\lceil 200kJ_m \rceil$ where $\lceil \ell \rceil = \min \{ m \in \intz : \ell \leq m \}$ is the ceiling of $\ell$. We further define
\begin{equation*}
     U_{m,l_2}(\chi)=
         \begin{cases}
             \Big( \sum_{j=0}^{J_m}\frac{1}{j!}\big( \Re D_{m,l_2}(\chi)\big)^j\Big)^2 & \text{ if }  n_m=0, \\
             e^{4 W_m}|D_{m,l_2}(\chi)W_m^{-1}|^{a_m}  & \text{ if } J_m/100k\leq W_m \leq 100kJ_m, \\
             \Big(2\frac{ (k-1)^{J_m} }{J_m!} (2W_m)^{J_m}\Big)^{\frac{2}{k-1} } |D_{m,l_2}(\chi) W_m^{-1}|^{a_m} & \text{ if } 100kJ_m\leq W_m.\\
         \end{cases}
\end{equation*}

  We quote the following result from \cite[Lemma 7]{Szab24}, which asserts that $U_{m,l_2}(\chi)$ dominates $R_{m,l_2}(\chi)^{1/(k-1)}$, up to a negligible error. Note that the proof of \cite[Lemma 7]{Szab24} depends only on the relative size of $|\Re R_{m,l_2}(\chi)|$ or $|R_{m,l_2}(\chi)|$ compared to $J_m$ or $W_m$ and is independent of the expression of $R_{m,l_2}(\chi)$, hence is valid for our case as well. 
\begin{lemma}
\label{easy}
With the notation as above, we have, for any $\chi$ mod $q$ and $1\leq m\leq M$, 
\begin{equation*}
       R_{m,l_2}(\chi)^{1/(k-1)} \leq \big(1+O(e^{-J_m})\big)U_{m,l_2}(\chi).
\end{equation*}
\end{lemma}

    We then deduce from the proof of \cite[Proposition 3.2]{Szab24} that in order to establish Proposition \ref{p2}, it suffices to prove the following result. 
\begin{proposition}
\label{big_upper}
    With the notation as above, we have, for fixed non-negative integers $n_1,\ldots n_M$,
\begin{equation*}
    \frac{1}{\varphi(q)}\sum_{\chi\in \mathcal{X}(n_1,\ldots,n_M)} \prod_{m=1}^M R_{m,l_1}(\chi)U_{m,l_2}(\chi)\ll \frac{(\log y)^{k^2}}{|l_1-l_2|^{2(k-1)}+1} \prod_{m=1}^M (\inf I^{(m)}_{n_m}+1)^{-2}.
\end{equation*}
\end{proposition}

  We now define $U_{m,l}(f)$ in the same way as we define $U_{m,l}(\chi)$, with $\chi$ being replaced by $f$. Recall that $R_{m,l}(f)$ is already defined in Section \ref{outline}. To establish Proposition \ref{big_upper}, we need two lemmas.  
\begin{lemma}
 \label{zerocase}
    With the notation as above, we have, for $U_{m,l_2}$ in the case where $n_m=0$, 
     \begin{equation*}
          \mathbb{E} \big|e^{2(k-1)\Re D_{m,l_1}(f)+2\Re D_{m,l_2}(f)}- R_{m,l_1}(f) U_{m,l_2}(f)\big|\leq e^{-J_m}.
     \end{equation*}
 \end{lemma}
\begin{proof}
   The proof is identical to that of \cite[Lemma 8]{Szab24} except the quantity $(k-1)^2 \sum_{y_{m-1}<p\leq y_m} \big(\frac{2}{p}+\frac{3}{p^2} \big)$ must be replaced by $4(k-1)^2 \sum_{y_{m-1}<p\leq y_m} \big(\frac{2}{p}+\frac{3}{p^2} \big)$. Then all we need is to verify that, with this new quantity, the bound $10^4(k-1)^2 \sum_{y_{m-1}<p\leq y_m} \big(\frac{2}{p}+\frac{3}{p^2} \big)\leq J_m$ holds. This follows from \eqref{estAJ} and so the lemma is proven.
\end{proof}
\begin{lemma}
\label{nonzerocase}
     With the notation as above, we have for $U_{m,l_2}$ in the case where $n_m\geq 1$, 
    \begin{equation*}
        \mathbb{E}  R_{m,l_1}(f) U_{m,l_2}(f)\leq (\inf I^{(m)}_{n_m}+1)^{-2}.
    \end{equation*}
\end{lemma}
\begin{proof}
  The proof is a straightforward modification of the proof of \cite[Lemma 9]{Szab24}. As shown there, we have
\begin{align}
\label{Esquareupperbound}
\begin{split}
    \big(\mathbb{E}  R_{m,l_1}(f) U_{m,l_2}(f) \big)^2\leq & \mathbb{E} R_{m,l_1}(f)^2 \mathbb{E}  U_{m,l_2}(f)^2 
\leq \exp\big( 4k^2A_m\big)\mathbb{E}  U_{m,l_2}(f)^2,
\end{split}
\end{align} 
  where $A_m$ is defined in \eqref{KAdef}. Similar to what is shown in the proof of \cite[Lemma 9]{Szab24}, upon using Lemma \ref{evenmoment}, we see that 
\begin{align}
\label{EUbound}
\begin{split}
    \mathbb{E}  U_{m,l_2}(f)^2\leq \begin{cases}
        \Big(\frac{e^2a_m A_m}{W^2}\Big)^{a_m} \text{  if  } J_m/100k \leq W\leq 100kJ_m, \\
        \Big( \frac{10ka_m A_m}{J_mW} \Big)^{a_m/2} \text{  if  } W\geq 100kJ_m. 
    \end{cases} 
\end{split}
\end{align}
  where $W=\inf I^{(m)}_{n_m}$. \newline

  Using \eqref{Abound} and the estimations $a_1 \leq 500kJ_1$, $W\geq \frac{J_1}{100k}$, $J_1=(\log \log y)^{3/2}$, together with the estimations $W\geq J_m/100k$, $a_m\leq 500kJ_m$, we see that 
\begin{align*}
\begin{split}
   \frac{e^2a_m A_m}{W^2} \leq \begin{cases}
        10^{10}k^3(\log \log y)^{-1/2}, & \text{  if  } m=1, \\
        2e^210^8k^3/J_m, & \text{  if  } m \geq 2. 
    \end{cases} 
\end{split}
\end{align*}
  Using $J_m \geq J_M \geq \exp(10^4k)$, we deduce from the above that for $y$ large enough, we have $\frac{e^2a_m A_m}{W^2}\leq e^{-1}$ for all $1 \leq m \leq M$.  Thus when $J_m/100k \leq W\leq 100kJ_m$, $\mathbb{E}U_{m,l_2}(f)^2\ll e^{-a_m}\leq e^{-W}$, and then as shown in the proof of \cite[Lemma 9]{Szab24}, we have $e^{4k^2A_m}e^{-W}\leq (W+1)^{-4}$. Similarly, we have $\mathbb{E}U_{m,l_2}(f)^2\ll W^{-100kJ_m}$ when $100kJ_m\leq W$ so that $e^{4k^2A_m}W^{-100kJ_m}\leq (W+1)^{-4}$. The assertion of the lemma now follows from \eqref{Esquareupperbound}, \eqref{EUbound} and the above discussion.  
\end{proof}

\begin{proof}[Proof of Proposition \ref{big_upper}]
     We deduce from \eqref{shortpolynomial} that $\prod_{m=1}^M R_{m,l_1}(\chi) U_{m,l_2}(\chi)$ is a Dirichlet polynomial of length not exceeding
\begin{align*}
\begin{split}
     \prod_{m=1}^M y_m^{8J_m+2a_m}<q.  
 \end{split}
\end{align*}
  Note further that $R_{m,l_1}(\chi) U_{m,l_2}(\chi)$ is non-negative. It follows from these and the orthogonality relation of Dirichlet character sums that
\begin{align}
\label{productbound}
\begin{split}
    \frac{1}{\varphi(q)} \sum_{\chi\in \mathcal{X}(n_1,\ldots,n_M)}\prod_{m=1}^M R_{m,l_1}(\chi) U_{m,l_2}(\chi) \leq & \frac{1}{\varphi(q)} \sum_{\chi\in \mathcal{X}_q }\prod_{m=1}^M R_{m,l_1}(\chi) U_{m,l_2}(\chi) \\
=&  \mathbb{E} \prod_{m=1}^M R_{m,l_1}(f) U_{m,l_2}(f) = \prod_{m=1}^M \mathbb{E}  R_{m,l_1}(f) U_{m,l_2}(f), 
 \end{split}
\end{align}
  where the last equality emerges by noting the random quantities $\big(R_{m,l_1}(f) U_{m,l_2}(f)\big)_{1\leq m\leq M}$ are independent of each other. \newline

 We deduce from Lemmas \ref{zerocase} and \ref{nonzerocase} that
\begin{equation}
\label{faszom2}
     \mathbb{E} R_{m,l_1}(f)U_{m,l_2}(f) \leq  \begin{cases}
         (\mathbb{E} e^{2(k-1)\Re D_{m,l_1}(f)+2\Re D_{m,l_2}(f)}+e^{-J_m})(\inf I^{(m)}_{n_m}+1)^{-2}, & \text{  if  } n_m=0, \\
        (\inf I^{(m)}_{n_m}+1)^{-2}, & \text{  if } n_m>0. 
    \end{cases}
\end{equation}
  Here we note that when $n_m=0$, the factor $(\inf I^{(m)}_{n_m}+1)^{-2}=1$.  So multiplying by it does not alter anything. \newline

 We next derive from Lemma \ref{eulerproduct} that
\begin{equation}
\label{faszom1}
\begin{split}
     & \mathbb{E} e^{2(k-1)\Re D_{m,l_1}(f)+2\Re D_{m,l_2}(f)} \\
     = & \exp \Big( O\Big(k\sum_{y_{m-1}<p\leq y_m } \frac{1}{p^{3/2} } \Big) \Big) \mathbb{E}\prod_{y_{m-1}<p\leq y_m} \Big|1-\frac{\alpha_pf(p)}{p^{1/2+il_1/\log y} } \Big|^{-2(k-1)} \\
     & \hspace*{2cm} \times  \Big|1-\frac{\beta_pf(p)}{p^{1/2+il_1/\log y} } \Big|^{-2(k-1)} \Big| 1-\frac{\alpha_pf(p)}{p^{1/2+il_2/\log y}}\Big|^{-2} \Big| 1-\frac{\beta_pf(p)}{p^{1/2+il_2/\log y}}\Big|^{-2} \\
     =&  \exp \Big( \sum_{y_{m-1}<p\leq y_m} \frac{(k-1)^2\lambda^2(p)}{p}+\frac{\lambda^2(p)}{p}+\frac{2(k-1)\lambda^2(p)\cos \big( \frac{l_1-l_2}{\log y} \log p\big)}{p}+O\Big( \frac{k^3}{y_{m-1}^{1/2}} +k\sum_{y_{m-1}<p\leq y_m } \frac{1}{p^{3/2} } \Big)\Big). 
\end{split}
\end{equation}
 
   It therefore follows from \eqref{faszom2}, \eqref{faszom1} that
\begin{align}
\label{Expectationbound}
\begin{split}
\prod_{m=1}^M \mathbb{E} R_{m,l_1}(f)U_{m,l_2}(f)  \leq &  \prod_{m=1}^M(\mathbb{E} e^{2(k-1)\Re D_{m,l_1}(f)+2\Re D_{m,l_2}(f)}+e^{-J_m})\prod_{m=1}^M (\inf I^{(m)}_{n_m}+1)^{-2} \\
\leq & \prod_{m=1}^M(1+e^{-J_m})\prod_{m=1}^M\mathbb{E} e^{2(k-1)\Re D_{m,l_1}(f)+2\Re D_{m,l_2}(f)}\prod_{m=1}^M (\inf I^{(m)}_{n_m}+1)^{-2} \\
\ll & \exp \Big( \sum_{p\leq y_M} \frac{(k-1)^2\lambda^2(p)}{p}+\frac{\lambda^2(p)}{p}+\frac{2(k-1)\lambda^2(p)\cos \big( \frac{l_1-l_2}{\log y} \log p\big)}{p} \\
& \hspace*{2cm} +O\Big( \sum_{m=1}^M \Big( \frac{k^3}{y_{m-1}^{1/2}} +k\sum_{y_{m-1}<p\leq y_m } \frac{1}{p^{3/2}} \Big) \Big)\Big) \prod_{m=1}^M (\inf I^{(m)}_{n_m}+1)^{-2}.
\end{split}
\end{align}

    Note that
\begin{equation}
\label{summerrorbound}
    \sum_{m=1}^M \Big(\frac{k^3}{y_{m-1}^{1/2}} +k\sum_{y_{m-1}<p\leq y_m } \frac{1}{p^{3/2} }\Big)\ll k^3.
\end{equation}

 Moreover, we deduce from \eqref{sumlambdapsquare}, \eqref{merten1} and Lemma \ref{RS3} that
\begin{equation}
\label{sumoverprimes}
    \exp \Big(\sum_{p\leq y_M} \frac{(k-1)^2\lambda^2(p)}{p}+\frac{\lambda^2(p)}{p}+\frac{2(k-1)\lambda^2(p)\cos \big( \frac{l_1-l_2}{\log y} \log p\big)}{p}\Big)\ll \frac{(\log y)^{k^2}}{|l_1-l_2|^{2(k-1)}+1},
\end{equation}
 The assertion of Proposition \ref{p2} now follows from \eqref{productbound}, \eqref{Expectationbound}--\eqref{sumoverprimes}. This completes the proof of the proposition. 
\end{proof}

\vspace*{.5cm}

\noindent{\bf Acknowledgments.}  P. G. is supported in part by NSFC grant 12471003 and L. Z. by the FRG Grant PS71536 at the University of New South Wales.  The authors would like to thank the anonymous referees for his/her very careful inspection of an earlier version of the paper and many helpful comments and suggestions.

\bibliography{biblio}

\begin{bibdiv}
\begin{biblist}

\bib{D}{article}{
      author={Deligne, P.},
       title={La conjecture de {Weil. I.}},
        date={1974},
     journal={Inst. Hautes Etudes Sci. Publ. Math.},
      volume={43},
       pages={273\ndash 307},
}

\bib{FI10}{book}{
      author={Friedlander, J.},
      author={Iwaniec, H.},
       title={Opera de cribro},
      series={American Mathematical Society Colloquium Publications},
   publisher={American Mathematical Society, Providence, RI},
        date={2010},
      volume={57},
}

\bib{GHH}{article}{
      author={Gao, P.},
      author={He, X.},
      author={Wu, X.},
       title={Bounds for moments of modular {$L$}-functions to a fixed
  modulus},
        date={2022},
     journal={Acta Arith.},
      volume={205},
      number={2},
       pages={161\ndash 189},
}

\bib{G&Zhao24-12}{article}{
      author={Gao, P.},
      author={Zhao, L.},
       title={Bounds for moments of twisted {F}ourier coefficients of modular
  forms},
        date={Preprint},
        note={arXiv:2412.12515},
}

\bib{harper2020moments}{article}{
      author={Harper, A.~J.},
       title={Moments of random multiplicative functions, {II}: {H}igh
  moments},
        date={2019},
     journal={Algebra Number Theory},
      volume={13},
      number={10},
       pages={2277\ndash 2321},
}

\bib{Harper}{article}{
      author={Harper, A.~J.},
       title={Sharp conditional bounds for moments of the {R}iemann zeta
  function},
        date={Preprint},
        note={arXiv:1305.4618},
}

\bib{Harper23}{article}{
      author={Harper, A.~J.},
       title={The typical size of character and zeta sums is $o(\sqrt{x})$},
        date={Preprint},
        note={arXiv:2301.04390},
}

\bib{iwakow}{book}{
      author={Iwaniec, H.},
      author={Kowalski, E.},
       title={{A}nalytic {N}umber {T}heory},
      series={American Mathematical Society Colloquium Publications},
   publisher={American Mathematical Society},
     address={Providence},
        date={2004},
      volume={53},
}

\bib{Kou}{article}{
      author={Koukoulopoulos, D.},
       title={Pretentious multiplicative functions and the prime number theorem
  for arithmetic progressions},
        date={2013},
     journal={Compos. Math.},
      volume={149},
      number={7},
       pages={1129\ndash 1149},
}

\bib{MVa1}{book}{
      author={Montgomery, H.~L.},
      author={Vaughan, R.~C.},
       title={{M}ultiplicative number theory. {I}. {C}lassical theory},
      series={Cambridge Studies in Advanced Mathematics},
   publisher={Cambridge University Press},
     address={Cambridge},
        date={2007},
      volume={97},
}

\bib{Munsch17}{article}{
      author={Munsch, M.},
       title={Shifted moments of {$L$}-functions and moments of theta
  functions},
        date={2017},
     journal={Mathematika},
      volume={63},
      number={1},
       pages={196\ndash 212},
}

\bib{Shimura}{article}{
      author={Shimura, G.},
       title={On the holomorphy of certain {D}irichlet series},
        date={1975},
     journal={Proc. London Math. Soc. (3)},
      volume={31},
      number={1},
       pages={79\ndash 98},
}

\bib{Szab}{article}{
      author={Szab\'o, B.},
       title={High moments of theta functions and character sums},
        date={2024},
     journal={Mathematika},
      volume={70},
      number={2},
       pages={Paper No. e12242, 37 pp.},
}

\bib{Szab24}{article}{
      author={Szab\'o, B.},
       title={A lower bound on high moments of character sums},
        date={Preprint},
        note={arXiv:2409.13436},
}

\end{biblist}
\end{bibdiv}
\bibliographystyle{amsxport}

\end{document}